\documentclass[preprint,12pt]{elsarticle}

\usepackage{comment}

\usepackage{graphicx,amsmath,color}
\usepackage{amssymb,amsthm,mathrsfs}
\usepackage{algorithmic,algorithm}
\usepackage{mathtools}
\usepackage{subfig}
\usepackage{hyperref}
\usepackage[normalem]{ulem}
\numberwithin{equation}{section}

\newcommand{\be}{\begin{equation}}
\newcommand{\ee}{\end{equation}}

\newcommand{\f}{\frac}

\newcommand{\bxi}{\boldsymbol \xi}

\newtheorem{lemma}{Lemma}[section]

\newtheorem{theorem}{Theorem}[section]

\graphicspath{{Fig/}}

\journal{Journal of Computational and Applied Mathematics}

\begin{document}

\begin{frontmatter}

\title{Anderson Acceleration Based on the $\mathcal{H}^{-s}$ Sobolev Norm for Contractive and Noncontractive Fixed-Point Operators}

\author[address1]{Yunan Yang\corref{mycorrespondingauthor}}
\ead{yy837@cornell.edu}
\cortext[mycorrespondingauthor]{Corresponding author}
\address[address1]{Department of Mathematics, Cornell University, Ithaca, NY 14850.}

\author[address1]{Alex Townsend}
\ead{townsend@cornell.edu}

\author[address3]{Daniel Appel\"{o}}
\ead{appeloda@msu.edu}
\address[address3]{Department of Computational Mathematics, Science, and Engineering, and Department of Mathematics, Michigan State University, East Lansing, MI 48824.}

\begin{abstract}
Anderson acceleration (AA) is a technique for accelerating the convergence of fixed-point iterations. In this paper, we apply AA to a sequence of functions and modify the norm in its internal optimization problem to the $\mathcal{H}^{-s}$ norm, for some positive integer $s$, to bias it towards low-frequency spectral content in the residual. We analyze the convergence of AA by quantifying its improvement over Picard iteration. We find that AA based on the $\mathcal{H}^{-2}$ norm is well-suited to solve fixed-point operators derived from second-order elliptic differential operators, including the Helmholtz equation.
\end{abstract}


\end{frontmatter}


\section{Introduction} \label{sec:intro}
Anderson acceleration (AA) or Anderson mixing is an acceleration method for fixed-point iterations. Given a continuous operator $G: \mathcal{X}\rightarrow \mathcal{X}$, where $\mathcal{X}\subseteq L^2(\Omega)$ is a Hilbert space and $\Omega \subseteq\mathbb{R}^n$, a basic method for finding a fixed-point of $G$, i.e., $x = G(x)$, is Picard iteration
\be \label{eq:Picard}
 x_{k+1} = G(x_k), \quad k\geq 1, \qquad x_0 \in \mathcal{X} \text{ given}.
\ee
AA can be used to speed up the convergence of $x_0,x_1,\ldots,$ to a fixed-point of $G$, or even calculate a fixed point when the Picard iterates diverge~\cite[Thm.~4.1]{pollock2019anderson}.  While Picard iteration only uses the current iterate to calculate the next one, $x_{k+1}$ in AA is a weighted sum of the previous $\min(k,m)+1$ iterates and residuals, where $m$ is a memory parameter. The weighted sum is chosen so that it minimizes a linearized residual~\cite[(4.16)]{anderson1965iterative} in the next iteration. 

The application of AA includes flow problems~\cite{pollock2018anderson}, solving nonlinear radiation-diffusion equations~\cite{an2017anderson}, and accelerating certain optimization algorithms~\cite{peng2018anderson,fu2019anderson,mai2019anderson,li2018fast}. 
It is closely related to Pulay mixing~\cite{pulay1980convergence} and DIIS (direct inversion on the iterative subspace)~\cite{kudin2002black,rohwedder2011analysis}, which are prominent methods in self-consistent field theory (SCFT)~\cite{Ceniceros:2004uq,PhysRevLett.83.4317}. AA is also becoming popular in the numerical analysis community~\cite{walker2011anderson,toth2015convergence,evans2018proof,zhang2018globally,pollock2019anderson}.  AA is related to many other iterative and acceleration methods. When $m = 0$, AA collapses to Picard iteration and when $m=\infty$, AA is essentially equivalent to GMRES (for Generalized Minimal RESidual) when the fixed-point operator is linear~\cite{walker2011anderson}. For any $m$, AA can be viewed as a multisecant quasi-Newton method~\cite{fang2009two,lin2013elliptic} and is also related to traditional series acceleration methods~\cite{weniger2001nonlinear}. 

As in Anderson's paper~\cite{anderson1965iterative}, we primarily regard AA as an iteration performed on functions. As an acceleration method on functions, calculating the weighted sum to minimize the linearized residual involves an optimization problem that is often posed in the $L^2$ norm. The $L^2$ norm is convenient because the optimization problem is then a continuous least-squares problem. For practical computations, AA on functions must be discretized so that functions become vectors, and the $L^2$ norm becomes the discrete $\ell^2$ norm. After discretization, the optimization problem becomes a classical least-squares problem, which can be solved using fast rank-updated QR factorizations~\cite[Sec.~6.5.1]{golub1996matrix}. 

We follow a casual suggestion by Anderson~\cite[p.~554]{anderson1965iterative} and minimize the linearized residual at each iteration in a norm other than the $L^2$ norm. In particular, we seek further acceleration for fixed-point iterations involving second-order elliptic differential operators by selecting a $\mathcal{H}^{-s}$ norm (see Section~\ref{sec:new_aa}). After discretization, any $\mathcal{H}^{-s}$ norm becomes a weighted least-squares norm. In certain situations, this can provide an implicit spectral bias to counterbalance the spectral biasing from a fixed-point operator. Other researchers have been motivated to modify the norm in AA based on the contraction properties, as opposed to spectral biasing, of $G$~\cite{pollock2018anderson}.

AA based on the $\mathcal{H}^{-s}$ norm is equivalent to a multisecant method in a weighted Frobenius norm (see Section~\ref{sec:multisecant}). We use this viewpoint to analyze its convergence behavior by comparing it to AA based on the $L^2$ norm and Picard iteration (see Section~\ref{sec:H-1-AA}). The improvement depends on the particular properties of the fixed-point operator. We present some analysis motivating the choice of the norm (see Theorem~\ref{thm:WeightedAAError}) as well as providing numerical experiments to demonstrate the benefit (see Section~\ref{sec:experiments}). 

The paper is structured as follows. In Section~\ref{sec:aa}, we provide background details on the AA method, introduce AA based on the $\mathcal{H}^{-s}$ norm, and show how it can be discretized. In Section~\ref{sec:analysis}, we give a detailed analysis of the error reduction achieved by performing one step of the AA iteration (see Theorem~\ref{thm:ErrorOneStepAA}).  In Section~\ref{sec:experiments}, we provide numerical experiments, including both contractive and noncontractive fixed-point operators for the 1D Poisson equation and linear and non-linear Helmholtz equations. Concluding remarks can be found in Section~\ref{sec:Conclusions}.

\section{Anderson acceleration} \label{sec:aa}
The convergence of Picard iteration in~\eqref{eq:Picard} is only guaranteed when certain assumptions hold on $G$ as well as the initial iterate $x_0$, and even then, its convergence is typically linear~\cite[Chap.~4.2]{kelley1995iterative}. To promote faster convergence, AA computes $x_{k+1}$ using the previous $\min(k,m) + 1$ iterates and residuals. The original form of AA is given in Algorithm~\ref{alg:AA1}, and the main step is to take a linear combination of the past $\min(k,m) + 1$ iterates to minimize a linearized residual of the form $G(x_i) - x_i$. 

\begin{algorithm}[ht!]
  \caption{The basic AA technique with the $L^2$ norm\label{alg:AA1}}
\begin{algorithmic}
\STATE {\bf Input:} Given $x_0\in\mathcal{X}$, mixing parameters $0\leq \beta_k\leq 1$, and memory parameter $m \ge 1$, this algorithm computes a sequence $x_0,x_1,\ldots,$ intended to converge to a fixed-point of $G:\mathcal{X}\rightarrow \mathcal{X}$. 
\FOR{$k = 0, 1, \ldots$ until convergence}
       \STATE $m_k = \min(m,k)$.
    \STATE Compute $F_k = (f_{k-m_k},\ldots,f_k)$, where $f_i = G(x_i)-x_i$. 
\STATE Solve 
\[
\alpha^{(k)} = {\rm arg}\!\!\!\!\!\!\!\!\!\!\!\!\min_{v\in\mathbb{C}^{m_k+1}, \sum_{i=0}^{m_k} v_i = 1} \|F_k v\|_{L^2}, \qquad \alpha^{(k)} = (\alpha^{(k)}_0,\ldots,\alpha^{(k)}_{m_k})^T. 
\]
\STATE Set
$$
x_{k+1} = (1-\beta_k) \sum_{i=0}^{m_k} \alpha_i^{(k)} x_{k-m_k+i}
 + \beta_k \sum_{i=0}^{m_k} \alpha_i^{(k)} G(x_{k-m_k+i}). 
$$ 
\ENDFOR
    \end{algorithmic}
\end{algorithm}

\begin{algorithm}[ht!]
  \caption{A reformulated AA technique for functions\label{alg:AA2}}
  \begin{algorithmic}
  \STATE {\bf Input:} Given $x_0\in\mathcal{X}$, memory parameter $m \ge 1$, and measure of distance $d : \mathcal{X}\times \mathcal{X} \rightarrow [0,\infty)$, this algorithm computes a sequence $x_0,x_1,\ldots,$ intended to converge to a fixed-point of $G:\mathcal{X}\rightarrow \mathcal{X}$.
  \STATE Set $x_1 = G(x_0)$.
\FOR{$k = 1, 2, \ldots$ until convergence}
       \STATE $m_k = \min(m,k)$.
    \STATE Set $D_k = (\Delta f_{k-m_k},\ldots,\Delta f_{k-1})$, where $\Delta f_i = f_{i+1}-f_i$ and $f_i$ = $G(x_i) - x_i$.
\STATE Solve 
\begin{equation}
\gamma^{(k)} = {\rm arg}\!\!\min_{v\in\mathbb{C}^{m_k}} d( f_k, D_k v ), \qquad \gamma^{(k)} = (\gamma^{(k)}_0,\ldots,\gamma^{(k)}_{m_k-1})^T. \label{eq:distance1}
\end{equation}
\STATE Set
\begin{equation}\label{eq:AA2_update}
x_{k+1}  = G(x_k) -  \sum_{i=0}^{m_k-1}  \gamma_i^{(k)} \left[G(x_{k-m_k+i+1}) - G(x_{k-m_k+i}) \right]. 
\end{equation}
\ENDFOR
    \end{algorithmic}
\end{algorithm}

The mixing parameters $\beta_k$ at iteration $k$ indicates how to combine the previous $m_k + 1$ iterates and residuals. The usual choice is to select $\beta_k=\beta$ for $k\geq 0$. AA with mixing parameter $\beta_k = \beta$ is the same as applying AA with $\beta_k = 1$ to the map $G_\beta(x) = (1-\beta)x + \beta G(x)$~\cite[p.~256]{kelley2018numerical}. Therefore, throughout this paper we use $\beta=1$. 

In Algorithm~\ref{alg:AA1}, the coefficient vector $\alpha^{(k)}$ is determined by a constrained optimization problem. To remove the constraint, and gain additional insight, one can set~\cite{walker2011anderson}
\[
\gamma_i^{(k)} = \alpha_0^{(k)} + \cdots + \alpha_i^{(k)}, \qquad 0\leq i\leq m_k-1. 
\]
By carefully rewriting Algorithm~\ref{alg:AA1} with $\beta_k = 1$ in terms of $\gamma_i^{(k)}$ for $0\leq i\leq m_k-1$, we obtain Algorithm~\ref{alg:AA2}. In this version of the algorithm, the coefficient vector $\gamma^{(k)}$ is determined by an unconstrained optimization problem, which can be computationally more convenient.

In Algorithm~\ref{alg:AA2}, one has a choice on $d: \mathcal{X}\times \mathcal{X} \rightarrow [0,\infty)$, which can be chosen as any distance function. A standard choice is to take $d(f,g) = \|f-g\|_{L^2(\mathbb{R}^n)}$ so that~\eqref{eq:distance1} can be efficiently solved. To see this, note that~\eqref{eq:distance1} becomes $\gamma^{(k)} = {\rm arg}\!\min_{v\in\mathbb{C}^{m_k}} \|f_k - D_k v\|_{L^2(\mathbb{R}^n)}$. This means that, when $D_k$ has linearly independent columns, we have
\begin{equation} 
\gamma^{(k)} = (D_k^* D_k)^{-1}D_k^*f_k, \qquad (D_k^* D_k)_{ij} = \langle \Delta f_i,\Delta f_j \rangle, \quad (D_k^*f_k)_i = \langle\Delta f_i, f_k \rangle,
\label{eq:ContinuousL2} 
\end{equation} 
where $\langle\cdot,\cdot\rangle$ is the standard $L^2$ inner-product. Moreover,~\eqref{eq:ContinuousL2} can be efficiently solved by a fast rank-updated QR factorization of quasimatrices~\cite{trefethen2009householder}.   

When $m$ is finite, AA is distinct from restarted GMRES in that it gradually phases out old residuals in favor of new ones while GMRES completely discards the history of the iterates every $m$ iterations. We demonstrate, by numerical experiments with the Helmholtz equation, that the gradual replacement strategy used by AA exhibits better convergence properties than restarted GMRES (see Section~\ref{sec:whi}). 

The computational efficiency and convergence rate of AA is affected by the distance function.
The majority of the literature focuses on the convergence of AA for vectors under the discrete $\ell^2$ norm (Euclidean distance). In this setting, AA is known to have superlinear convergence~\cite{toth2015convergence,pollock2018anderson} when accelerating fixed-point contraction operators. AA can also converge for sequences from noncontractive fixed-point operators~\cite{pollock2019anderson}.

\subsection{The Hilbert space $\mathcal{H}^{-s}$ and its norm} \label{sec:new_aa}
We select the distance function in Algorithm~\ref{alg:AA2} to be $d(f,g) = \|f-g\|_{\mathcal{H}^{-s}}$ for some positive integer $s$. We observe that this can speed up the convergence of AA for fixed-point operators defined via second-order elliptic differential operators (see Section~\ref{sec:experiments}).

One can define $\mathcal{H}^{-s}(\mathbb{R}^n)$, for any real number $s$, as
\[
\mathcal{H}^{-s}(\mathbb{R}^n) = \left\{f\in \mathcal{S}'(\mathbb{R}^n) :\mathcal{F}^{-1}\left[ (1+|\bxi|^2)^{-s/2}\mathcal{F}f\right]\in L^2(\mathbb{R}^n)\right\}, 
\]
where $\mathcal{F}$ denotes the Fourier transform on $\mathbb{R}^n$ and $\mathcal{S}'(\mathbb{R}^n)$ is the space of tempered distributions~\cite{leoni2017first}. The Hilbert space $\mathcal{H}^{-s}(\mathbb{R}^n)$ can be equipped with the norm 
\begin{equation} 
\|f\|_{\mathcal{H}^{-s}(\mathbb{R}^n)} = \left\|\mathcal{F}^{-1}\left[(1+|\bxi|^2)^{-s/2}\mathcal{F}f\right]\right\|_{L^2(\mathbb{R}^n)}, \qquad f\in\mathcal{H}^{-s}(\mathbb{R}^n). 
\label{eq:HilbertSpaceNorm}
\end{equation} 
It turns out that the solution to~\eqref{eq:distance1} with $d(f,g) = \|f-g\|_{\mathcal{H}^{-s}(\mathbb{R}^n)}$ can be expressed using a weighted projection formula. To see this, note that~\eqref{eq:distance1} becomes
\[
\gamma^{(k)} = {\rm arg}\!\!\min_{v\in\mathbb{C}^{m_k}}\left\| f_k-v\right\|_{\mathcal{H}^{-s}(\mathbb{R}^n)} = {\rm arg}\!\!\min_{v\in\mathbb{C}^{m_k}}\left\| \mathcal{P}(f_k-v)\right\|_{L^2(\mathbb{R}^n)}, \qquad \mathcal{P}f = \mathcal{F}^{-1}\left[(1+|\bxi|^2)^{-s/2}\mathcal{F}f\right]
\]
and hence, when the columns of $D_k$ are linearly independent, we have
\begin{equation} 
\gamma^{(k)} = \left(D_k^* \mathcal{P}^2  D_k\right)^{-1} D_k^*\mathcal{P}^2 f_k.
\label{eq:HsSolution} 
\end{equation} 
Here, we have $(D_k^* \mathcal{P}^2  D_k)_{ij} = \langle\mathcal{P}\Delta f_{k-m_k+i},\mathcal{P}\Delta f_{k-m_k+j} \rangle$ for $0\leq i,j\leq m_k-1$ and $(D_k^*\mathcal{P}^2 f_k)_i = \langle \mathcal{P}\Delta f_{k-m_k+i}, \mathcal{P}f_k \rangle$ for $0\leq i\leq m_k-1$.

The Hilbert space $\mathcal{H}^{-s}(\Omega)$ for a bounded Lipschitz-smooth domain $\Omega \subseteq\mathbb{R}^n$ is the set of restrictions of functions from $\mathcal{H}^{-s}(\mathbb{R}^n)$ equipped with the norm 
\[
\|f\|_{\mathcal{H}^{-s}(\Omega)} = \inf\left\{\|g\|_{\mathcal{H}^{-s}(\mathbb{R}^n)} : g\in\mathcal{H}^{-s}(\mathbb{R}^n), g|_\Omega = f\right\}.
\]
An equivalent and more explicit definition of $\|f\|_{\mathcal{H}^{-s}(\Omega)}$ is given via the Laplacian operator~\cite[p.~586]{schechter1960negative}. That is, when $\Omega$ is a bounded domain with infinitely differentiable boundary, we have
\begin{equation} 
\|f\|_{\mathcal{H}^{-s}(\Omega)} =  \|u\|_{\mathcal{H}^{s}(\Omega)}, 
\label{eq:Hm1norm}
\end{equation} 
where $u$ is the solution to $(\sum_{r=0}^s(-1)^r\nabla^{2r}) u=f$ with $u$ satisfying zero Neumann boundary conditions. 

We can begin to appreciate the purpose of the $\mathcal H^{-s}$ norm from~\eqref{eq:HilbertSpaceNorm} and~\eqref{eq:HsSolution}. The norm $\|f\|_{\mathcal{H}^{-s}}$ weights the low-frequency spectral content of $f$ more than the high-frequency content. Thus, $\gamma^{(k)}$ is focused on making the low-frequency spectral content of the residual smaller, which can potentially counterbalancing the spectral biasing of a fixed-point operator.  One can select any type of frequency biasing---towards the low- or high-frequency spectral content of the residual---by choosing $d(f,g) = \|f-g\|_{\mathcal H^{-s}}$ for $s\in\mathbb{R}$. We suspect that a good choice of $s$ depends on the spectral biasing of the fixed-point operator (see Section~\ref{sec:example1}). While we focus on the benefits of the $\mathcal H^{-s}$ norm, the idea of acceleration through changing the distance function is more general. One can select the distance function a priori or even modify it dynamically as the iteration proceeds. 

The choice of distance function in AA is similar to preconditioning in an iterative method. The work in this paper began with the idea that integration reformulation~\cite{driscoll2010automatic,greengard1991spectral} or integral preconditioning~\cite{hesthaven1998integration} might be a way to speed up the convergence of AA for certain fixed-point operators. We suspect that there are possible connections of our work to operator preconditioning~\cite{hiptmair2006operator}, continuous Krylov methods~\cite{gilles2019continuous}, and Riesz operators~\cite{malek2014preconditioning}.


\subsection{Discretizing Anderson acceleration based on the $\mathcal{H}^{-s}$ norm}\label{sec:AAdiscretized}
One must first discretize Algorithm~\ref{alg:AA2} before running it on a computer. In principle, any reasonable discretization scheme is appropriate. In this paper, we discretize functions and operators with finite difference schemes so that the iterates $x_k\in\mathcal{X}$ in Algorithm~\ref{alg:AA2} are replaced by vectors that sample $x_k$ at equispaced points. 

The most subtle quantity to discretize in Algorithm~\ref{alg:AA2} is $d(f,g)$. 
For example, if $\Omega = (0,1)$, then a discrete analogue of $d(f,g) = \|f-g\|_{\mathcal{H}^{-1}(\Omega)}$ is given by
\begin{equation}
d(v, w) = \sqrt{h}\left\|(I_n - B_n)^{-1/2}(v-w)\right\|_2, \qquad B_n = \frac{1}{h^2}\begin{bmatrix}-1 & 1 \cr 1 & -2 & \ddots \cr & \ddots & \ddots & 1\cr & & 1 & -2 & 1\cr & & & 1 & -1 \end{bmatrix}, \label{eq:Hm1Matrix}
\end{equation}
where $h = 1/(n-1)$. Here, $I_n$ is the $n \times n$ identity matrix and $B_n$ is the $n\times n$ second-order finite difference matrix for the Laplacian with zero Neumann conditions. 
From~\eqref{eq:Hm1norm}, we know that $(d(f,g))^2 = \|u\|_{\mathcal{H}^1(\Omega)}^2 = \langle u, u\rangle + \langle \nabla u, \nabla u\rangle = \langle u, u-\nabla^2 u\rangle$. Therefore, we take the discrete analogue as $d(v,w) = \sqrt{h}\|(I_n-B_n)^{1/2} (I_n-B_n)^{-1}(v-w)\|_2 = \sqrt{h}\|(I_n-B_n)^{-1/2} (v-w)\|_2$. Here, the integral in the definition of the $L^2$ norm is discretized by a low-order Riemann-like sum.

We have selected the so-called half-sample discretization for $u'(0)=0$ and $u'(1)=0$ so that the matrix is symmetric~\cite[Sec.~3]{strang1999discrete}. For $d(f,g) = \|f-g\|_{\mathcal{H}^{-2}(\Omega)}$, similar to~\eqref{eq:Hm1Matrix}, we take the discretization as
\[
d(v, w) = \sqrt{h}\|(I_n-B_n+B_n^2)^{-1/2}(v-w)\|_2.
\]
Therefore, once AA in an $\mathcal{H}^{-s}$ norm is discretized for a nonzero $s$, it becomes AA in a weighted $\ell^2$ norm. 

\subsection{The connection with the multisecant method}\label{sec:multisecant} 
A useful interpretation of AA for vectors is as a multisecant method~\cite{fang2009two,pratapa2016anderson}. In particular, when $d(v,w) = \|v-w\|_2$, the update in~\eqref{eq:AA2_update} can be expressed in the following form: 
\[
x_{k+1} = x_{k}  + (I-S_k) f_k,
\]
where $S_k \in \mathbb{C}^{n\times n}$, $x_{k}$ is the $k$th iterate from AA, and $f_k=G(x_k)-x_k$. It is shown in~\cite{fang2009two} that if one defines $\Delta x_i = x_{i+1}-x_i$, $\Delta f_i = f_{i+1}-f_i$, and
\begin{equation} 
X_k = \begin{bmatrix} \Delta x_{k-m_k},\ldots,\Delta x_{k-1}\end{bmatrix}\in\mathbb{C}^{n\times m_k},\quad D_k = \begin{bmatrix}\Delta f_{k-m_k},\ldots,\Delta f_{k-1}\end{bmatrix}\in\mathbb{C}^{n\times m_k}, 
\label{eq:XkDk}
\end{equation} 
then $S_k$ is the solution to the following constrained optimization problem:
\begin{equation} 
\min_{S_k\in\mathbb{C}^{n\times n}} \|S_k\|_{F}, \quad \text{subject to} \quad S_k D_k = X_k+D_k.
\label{eq:OptProb}
\end{equation} 
Here, $\|\,\cdot\,\|_F$ denotes the matrix Frobenius norm, i.e., $\|S_k\|_F^2 = \sum_{i,j=1}^n|(S_k)_{ij}|^2$. Furthermore,~\eqref{eq:OptProb} has an explicit solution given by 
\[
S_k = \left(X_k + D_k\right)\left(D_k^* D_k\right)^{-1}D_k^*,
\]
when the columns of $D_k$ are linearly independent~\cite{fang2009two,lin2013elliptic}.

The $\mathcal H^{-s}$ norm can be discretized to a weighted $\ell^2$ norm when $s$ is an integer, i.e., $d(v,w) = \|P(v-w)\|_2$, where $P$ is a symmetric positive definite matrix. For example, $P =(I_n-B_n)^{-1/2}$ for $\mathcal H^{-1}$ norm and $P=(I_n-B_n+B_n^2)^{-1/2}$ for $\mathcal H^{-2}$ norm. Similar to the $\ell^2$ norm, one can also write AA based on the $\mathcal H^{-s}$ norm as a multisecant method. From a weighted projection formula, we find that 
\[
x_{k+1} = x_{k}  + (I - \widetilde S_k) f_k,
\]
where $\widetilde S_k$ solves the following constrained optimization problem
\begin{equation} 
\min_{\widetilde S_k} ||P  \widetilde  S_k   P^{-1}||_F, \quad \text{subject to}\quad \widetilde S_k  D_k  =  X_k + D_k.
\label{eq:secant2}
\end{equation} 
Furthermore,~\eqref{eq:secant2} has an explicit solution given by 
\begin{equation} \label{eq:tilde_Sk}
\widetilde S_k = \left( X_k + D_k\right)\left( D_k^*P^2 D_k\right)^{-1}D_k^* P^2,
\end{equation}
which is derived in Lemma~\ref{lem:tilde_Sk} below. 

\begin{lemma} \label{lem:tilde_Sk}
Let $P$ be a positive definite matrix and let $D_k,X_k\in\mathbb{C}^{n\times m_k}$ such that $D_k$ has linearly independent columns. Then, the solution to
\[
\min_{S_k} ||S_k||_{F}, \quad \text{subject to}\quad S_k PD_k = PX_k + PD_k,
\]
is $S_k = P \widetilde S_k P^{-1}$, where $\widetilde S_k = \left( X_k + D_k\right)\left( D_k^*P^2 D_k\right)^{-1}D_k^* P^2$ is the solution to~\eqref{eq:secant2}.
\end{lemma} 
\begin{proof} 
The solution to 
\[
\min_{S'_k} ||S'_k||_{F}, \quad \text{subject to}\quad S'_k D'_k = X'_k + D'_k,
\]
is given by $S'_k = \left(X'_k + D'_k\right)\left((D'_k)^* D'_k\right)^{-1}(D'_k)^*$~\cite{lin2013elliptic}. The statement of the lemma follows by setting $D'_k = P D_k$ and $X'_k = P X_k$.
\end{proof} 

The interpretation of AA based on the $\mathcal{H}^{-s}$ norm as a multisecant method is useful for understanding its convergence behavior in Section~\ref{sec:analysis}. 

\section{Error analysis of Anderson acceleration}\label{sec:analysis}
In this section, we analyze AA based on the $\mathcal{H}^{-2}$ norm. Suppose that we have the following fixed-point iteration
\[
x_{k+1} = G(x_k) = Ax_k+b, \qquad k\geq 0,
\]
where $A$ is an $n \times n$ real symmetric matrix and $b$ is an $n\times 1$ vector. Let $x^*$ denote a fixed-point, i.e., $G(x^*) = x^*$. The dependence between the error $e_k = x_k - x^*$ and the residual $f_k = G(x_k) - x_k$ in any two consecutive iterates can be written as
\[
e_{k+1}= A e_k,\qquad f_{k+1} = Af_k, \qquad k\geq 0.
\]
Since $A$ is a real symmetric matrix, it has an orthogonal eigendecomposition given by 
\[
A = W \Lambda W^*, \qquad \Lambda = {\rm diag}(\lambda_1,\ldots,\lambda_n),
\]
where $W$ is an orthogonal matrix. After $k$ Picard iterations, $e_{k} = A^{k}e_0 = W\Lambda^{k} W^*e_0$ for $k\geq 0$, where $e_0 = x_0 - x^*$ is the initial error. There is an extensive literature on the convergence of AA~\cite{toth2015convergence,pollock2019anderson,pollock2018anderson}, but an explicit convergence rate that depends on $m$ is missing. It is nontrivial to derive an explicit convergence rate as a function of $m$ since in every iteration the weights in AA are derived from an optimization problem. 

In this paper, we are particularly interested in the relation between the convergence of AA and the memory parameter as well as the choice of distance function in~\eqref{eq:distance1}. The setting of our analysis is the following. We first apply Picard iteration for $k$ iterations (where $k\geq m$), and then we perform one step of the AA algorithm with memory parameter $m$ to obtain the $(k+1)$th iterate. We analyze the solution error $e_{k+1}$ after one step of AA and compare it to the solution error after $k+1$ Picard iterations. We refer to this as the \textit{one-step} analysis of AA. The improvement in the solution error is the \textit{one-step} acceleration of AA.

Alternating between Picard and AA is proposed as the alternating Anderson--Jacobi method in~\cite{pratapa2016anderson}. Although it is convenient for the analysis, we do not advocate using an alternating scheme in practice since, in our experience, applying AA at every iteration usually has an improved convergence behavior.

\subsection{Error analysis of Anderson acceleration in the $\ell^2$ norm}\label{sec:L2-AA}
In Section~\ref{sec:multisecant}, AA is viewed as a multisecant-type method with iterates defined as
\begin{equation} \label{eq:Sk}
x_{k+1} = x_k + (I-S_k)f_k, \qquad S_k = (D_k+X_k)(D_k^*D_k)^{-1}D_k^*,
\end{equation}
where the matrices $D_k$ and $X_k$ are defined in~\eqref{eq:XkDk} and $f_k = G(x_k) - x_k$. We note that Picard iteration takes $S_k$ to be the zero matrix, and we hope that~\eqref{eq:Sk} promotes faster convergence to a fixed-point of $G$. The error and residual between any two consecutive iterates satisfy the following recurrence:
\begin{equation} \label{eq:err_res_AA}
e_{k+1}= \left( I - (I-S_k)(I-A)\right) e_k,\quad f_{k+1} = \left( I - (I-A)(I-S_k) \right) f_k.
\end{equation} 

The matrix $S_k$ in~\eqref{eq:err_res_AA} depends on the fixed-point operator as well as the previous $m_k$ iterates and residuals. Therefore, we find it difficult to imagine a full and explicit convergence analysis of AA for general $m$. Instead, we analyze the acceleration effect when one runs Picard iteration for the first $k$ iterations (where $k\geq m$) and then performs one step of AA to obtain the $(k+1)$th iterate. We start our error analysis by expressing the solution error explicitly in terms of a Krylov matrix. For an $n\times n$ matrix $A$ and an $n\times 1$ vector $b$, a Krylov matrix is defined as 
\[
K_m(A,b) = \begin{bmatrix}b&Ab&\cdots&A^{m-1}b\end{bmatrix}\in\mathbb{C}^{n\times m}.
\]

We have the following lemma that relates the error after doing Picard for $k+1$ steps, denoted by $Ae_{k}$, to the error after doing Picard for $k$ steps and then one step of AA with memory parameter $m$, denoted by $e_{k+1}$. 
\begin{lemma}~\label{lem:OneStepAA}
Let $A$ be an $n\times n$ real symmetric matrix with eigenvalue decomposition $A = W\Lambda W^*$ and $b$ be an $n\times 1$ vector.  Suppose that $x_{j+1} = Ax_j + b$ for $0\leq j\leq k-1$, $0$ and $1$ are not eigenvalues of $A$, and that $x_{k+1}$ is produced from AA based on the discrete $\ell^2$ norm with memory parameter $1\leq m\leq n$. Then, when $K_H$ has linearly independent columns, we have
\be \label{eq:ek}
e_{k+1} = W E W^* A e_k, \qquad E = D_{\mu} \left(I - K_H\left(K_H^*K_H\right)^{-1}K_H^*\right)D_{\mu}^{-1},
\ee
where $D_{\mu} = \Lambda(\Lambda-I)^{-1}$, $K_H = K_m(\Lambda,HW^*e_0)$, and $H = (\Lambda-I)^2\Lambda^{k-m}$. 
\end{lemma}
\begin{proof} 
Note that $e_{j+1} - e_j = (A^{j+1}-A^j)e_0 = W\Lambda^j(\Lambda - I)W^*e_0$ and $f_{j+1} - f_j = (A-I)(e_{j+1} - e_j)$ for $0\leq j\leq k-1$. Thus, for $D_k$ and $X_k$ given in~\eqref{eq:XkDk}, we have
\begin{equation} 
D_k = W\Lambda^{k-m}(\Lambda - I)^2 K_m(\Lambda, W^*e_0)
\label{eq:Dk}
\end{equation} 
and $X_k = W\Lambda^{k-m}(\Lambda - I)K_m(\Lambda, W^*e_0)$. By substituting~\eqref{eq:Dk} into~\eqref{eq:Sk}, we find that 
\begin{equation} 
S_k = W\Lambda(\Lambda-I)^{-1}K_H\left(K_H^*K_H\right)^{-1}K_H^*, \qquad K_H = K_m(\Lambda,HW^*e_0),
\label{eq:FullSk}
\end{equation} 
where $H = (\Lambda-I)^2\Lambda^{k-m}$. Furthermore, by substituting~\eqref{eq:FullSk} into~\eqref{eq:err_res_AA}, we obtain the following: 
\[
e_{k+1} = W\Lambda(\Lambda-I)^{-1}\left(I - K_H\left(K_H^*K_H\right)^{-1}K_H^*\right)(\Lambda - I) W^* e_k. 
\]
The result follows by noting that $W^*e_k = W^*A^{-1}Ae_k = \Lambda^{-1} W^* A e_k$. 
\end{proof} 

The main observation from Lemma~\ref{lem:OneStepAA} is that $E$ is a projection matrix. Furthermore, $D_{\mu}^{-1}ED_{\mu}$ is an \textit{orthogonal} projection onto the space spanned by the column space of the Krylov matrix $K_H$. Since $H$ is known explicitly, one can precisely quantify the difference between $\|D_{\mu}^{-1}W^*Ae_{k}\|_2$ and $\|D_{\mu}^{-1}W^*e_{k+1}\|_2$. 

\begin{theorem}~\label{thm:ErrorOneStepAA}
Under the same setup, notation, and assumptions of Lemma~\ref{lem:OneStepAA}, suppose that the eigenvalues of $A$ are contained in an interval $[a,b]$ that does not contain $0$ or $1$. Then, 
\[
\|D_{\mu}^{-1} W^* e_{k+1}\|_2 \leq C(a,b,m)  \|D_{\mu}^{-1}W^*Ae_k\|_2, \qquad C(a,b,m) = \left| T_m\left(\frac{2ab-a-b}{b-a}\right)\right|^{-1},
\]
where $T_m(x)$ is the Chebyshev polynomial of degree $m$.
\end{theorem}
\begin{proof} 
From Lemma~\ref{lem:OneStepAA} and since $K_H$ is a Krylov matrix, we find that 
\[
\begin{aligned} 
\|D_{\mu}^{-1} W^* e_{k+1}\|_2 
& = \| (I - K_H\left(K_H^*K_H\right)^{-1}K_H^*) D_{\mu}^{-1}W^*Ae_k\|_2 \\
& = \min_{c\in\mathbb{C}^m} \| D_{\mu}^{-1}W^*Ae_k - K_Hc \|_2\\
& = \min_{p\in\mathcal{P}_{m-1}} \| D_{\mu}^{-1}W^*Ae_k - p( \Lambda )HW^*e_0 \|_2,
\end{aligned} 
\]
where $\mathcal{P}_{m-1}$ is the space of polynomials of degree $\leq m -1$. 
Since $H = (\Lambda - I)^2\Lambda^{k-m}$ and $D_{\mu}^{-1} = (\Lambda-I)\Lambda^{-1}$, we know that $HW^*e_0 = \Lambda^{k-m}(\Lambda - I)^2W^*e_0 = \Lambda^{-m}(\Lambda - I) D_{\mu}^{-1}W^*Ae_k$. Therefore, we find that 
\[
\begin{aligned} 
\|D_{\mu}^{-1} W^* e_{k+1}\|_2 
&\leq \min_{p\in\mathcal{P}_{m-1}} \| I - p( \Lambda )\Lambda^{-m}(\Lambda-I)\|_2 \|D_{\mu}^{-1}W^*Ae_k \|_2\\
& \leq \min_{p\in\mathcal{P}_{m-1}} \max_{x\in[a,b]} \left| 1 - p( x )x^{-m}(x-1)\right| \|D_{\mu}^{-1}W^*Ae_k \|_2.\\
\end{aligned} 
\]
We note that 
\[
\begin{aligned} 
 \min_{p\in\mathcal{P}_{m-1}} \max_{x\in[a,b]} \left| 1 - p( x )x^{-m}(x-1)\right| 
 & = \min_{p\in\mathcal{P}_{m-1}} \max_{x\in[\tfrac{1}{b},\tfrac{1}{a}]} \left| 1 - p( \tfrac{1}{x} )x^{m-1}(1-x)\right|\\
 & =  \min_{p\in\mathcal{P}_{m-1}} \max_{x\in[\tfrac{1}{b},\tfrac{1}{a}]} \left| 1 - p( x ) (1-x)\right|\\
 & = \min_{q\in\mathcal{P}_{m},q(1) = 1} \max_{x\in[\tfrac{1}{b},\tfrac{1}{a}]} \left| q(x)\right|. 
\end{aligned}
\]
For any $|x_*|>1$, we know that $|T_m(x_*)|\geq |p(x_*)|$ for any polynomial $p$ of degree $\leq m$ such that $|p(x)|\leq 1$ for $x\in[-1,1]$, where $T_m(x)$ is the Chebyshev polynomial of degree $m$~\cite{trefethen1997numerical}. Therefore, since $1\not\in [1/b,1/a]$, we have
\[
 \min_{q\in\mathcal{P}_{m},q(1) = 1} \max_{x\in[\tfrac{1}{b},\tfrac{1}{a}]} \left| q(x)\right| = \max_{x\in[\tfrac{1}{b},\tfrac{1}{a}]}\left|\frac{T_m\left(\frac{2(x-\tfrac{1}{b})}{\tfrac{1}{a}-\tfrac{1}{b}} - 1\right)}{T_m\left(\frac{2(1-\tfrac{1}{b})}{\tfrac{1}{a}-\tfrac{1}{b}} - 1\right)}\right|\leq \left|T_m\left(\frac{2(1-\tfrac{1}{b})}{\tfrac{1}{a}-\tfrac{1}{b}} - 1\right)\right|^{-1},
\]
where in the last inequality we used the fact that $|T_m(x)|\leq 1$ for $x\in[-1,1]$. The result now follows as $2(1-1/b)/(1/a-1/b)-1 = (2ab - a - b)/(b-a)$. 
\end{proof}

The vector given by $Ae_k$ is the solution error after $k+1$ Picard iterations, while $e_{k+1}$ is the solution error after $k$ Picard iterations and then one step of AA. Therefore, Theorem~\ref{thm:ErrorOneStepAA} provides a bound on the acceleration effect by performing one-step AA. If $0<a<b<1$, the weighting matrix $D^{-1}_{\mu} = \Lambda(I-\Lambda)^{-1}$ enforces large weights on components of the error related to eigenvalues that are close to 1. Therefore, this particular weighting suggests that one-step AA is improving precisely the components of the residual that are making Picard iteration converge slowly. 

The number $C(a,b,m)$ in Theorem~\ref{thm:ErrorOneStepAA} only depends on $a$, $b$, and $m$, where $[a,b]$ is an interval containing the eigenvalues of the fixed-point iteration matrix and $m$ is the memory parameter in AA. For example, $C(0.3,0.9,10) \leq 0.024$ and $C(2,100,10) \leq 3.84\times 10^{-8}$. For any interval $[a,b]$, not containing $0$ and $1$, the number $|(2ab-a-b)/(b-a)|>1$, and hence $C(a,b,m)$ is a monotonically decreasing function of $m$ (for fixed $a$ and $b$). In fact, as a function of $m$, $C(a,b,m)$ decays exponentially to zero as $m\rightarrow \infty$. 

\subsection{Error analysis of Anderson acceleration in a weighted $\ell^2$ norm}\label{sec:H-1-AA}
One can derive explicit formulas for AA when performed in a weighted $\ell^2$ norm. That is, the distance function in~\eqref{eq:distance1} is $d(v,w) = \|P(v-w)\|_2$ for some positive definite matrix $P$. From~\eqref{eq:tilde_Sk}, we find that
\[
\tilde{S}_k = (X_k + D_k)(D_k^*P^2D_k)^{-1}D_k^*P^2.
\]
Since $D_k = W\Lambda^{k-m}(\Lambda - I)^2 K_m(\Lambda, W^*e_0)$ and $X_k = W\Lambda^{k-m}(\Lambda - I)K_m(\Lambda, W^*e_0)$, we have
\begin{equation} 
\tilde{S}_k = JWK_H(K_H^*W^*P^2WK_H)^{-1}K_H^*W^*P^2,
\label{eq:Hm1_AA}
\end{equation} 
where $H = (\Lambda-I)^2\Lambda^{k-m}$, $J = W\Lambda(\Lambda-I)^{-1}W^*$, and $K_H = K_m(\Lambda,HW^*e_0)$. Equation~\eqref{eq:Hm1_AA} allows us to derive an analogous formula to~\eqref{eq:ek} for AA in a weighted $\ell^2$ norm. 

\begin{lemma}~\label{lem:WeightedAA} 
Under the same setup, notation, and assumptions of Lemma~\ref{lem:OneStepAA}, except that $x_{k+1}$ is produced from AA with $d(v,w) = \|P(v-w)\|_2$ for some positive definite matrix. Then, we have
\[
e_{k+1} = W \tilde{E} W^* Ae_k, \qquad \tilde{E} = D_{\mu}\left[I - K_H(K_H^*W^*P^2WK_H)^{-1}K_H^*W^*P^2 W \right] \! D_{\mu}^{-1}. 
\]
\end{lemma} 
\begin{proof} 
The proof is essentially identical to the proof of Lemma~\ref{lem:OneStepAA}.
\end{proof} 

Here, $\tilde{E}$ is a projection matrix and hence we know that $||e_{k+1}||_2 \leq||Ae_{k}||_2$. However, now $D_{\mu}^{-1}\tilde{E}D_{\mu}$ is not an orthogonal projection (unless $P = I$). This makes it very difficult to do the analysis of one-step AA with a weighted $\ell^2$ norm. 

\subsubsection{When $P$ and $A$ share the same eigenvectors}
To make progress here, we make a strong assumption that the fixed-point operator and the matrix $P$ share the same eigenvectors.  That is, we suppose that $A = W\Lambda W^*$ and $P = W\Sigma W^*$ for the same orthogonal matrix $W$.  Under this assumption, we can analyze the acceleration effect of one-step AA when performed with $d(v,w) = \|P(v-w)\|_2$. The following theorem is a generalization of  Theorem~\ref{thm:ErrorOneStepAA}.

\begin{theorem}~\label{thm:WeightedAAError} 
Under the same setup, notation, and assumptions as Lemma~\ref{lem:WeightedAA}, with eigenvalue decompositions $A = W\Lambda W^*$ and $P = W\Sigma W^*$, we have 
\[
\frac{\|\Sigma D_{\mu}^{-1}W^*e_{k+1}\|_2}{\max_{1\leq i\leq n} |\Sigma_{ii}|} \leq C(a,b,m)\|D_{\mu}^{-1}W^*Ae_k\|_2, \qquad C(a,b,m) = \left|T_m\left(\frac{2ab-a-b}{b-a}\right)\right|^{-1}.
\]
\end{theorem} 
\begin{proof} 
Since $P = W\Sigma W^*$, that statement of Lemma~\ref{lem:WeightedAA} becomes 
\[
e_{k+1} = W D_{\mu}\left[I - K_H(K_H^*\Sigma^2K_H)^{-1}K_H^*\Sigma^2 \right] \! D_{\mu}^{-1} W^* Ae_k. 
\]
Therefore, we have
\[
\begin{aligned} 
\|\Sigma D_{\mu}^{-1}W^*e_{k+1}\|_2 
& = \|\Sigma(I - K_H(K_H^*\Sigma^2K_H)^{-1}K_H^*\Sigma^2 )D_{\mu}^{-1} W^* Ae_k\|_2\\
& = \min_{c\in\mathbb{C}^m} \|\Sigma (D_{\mu}^{-1} W^* Ae_k - K_Hc)\|_2\\
& = \min_{p\in\mathcal{P}_{m-1}} \| \Sigma (D_{\mu}^{-1}W^*Ae_k - p( \Lambda )HW^*e_0) \|_2,
\end{aligned} 
\]
where the second equality follows from the formula for weighted projection~\cite[Sec.~6.1.1]{golub1996matrix}. Here, $\mathcal{P}_{m-1}$ is the space of polynomials of degree $\leq m-1$.
Since $HW^*e_0 = \Lambda^{-m}(\Lambda - I) D_{\mu}^{-1}W^*Ae_k$ (see the proof of Theorem~\ref{thm:ErrorOneStepAA}), we have 
\[
\begin{aligned} 
\|\Sigma D_{\mu}^{-1}W^*e_{k+1}\|_2 
& \leq  \min_{p\in\mathcal{P}_{m-1}} \| \Sigma (I - p( \Lambda )\Lambda^{-m}(\Lambda - I)) \|_2 \|D_{\mu}^{-1}W^*Ae_k\|_2\\
& \leq  \max_{1\leq i\leq n} |\Sigma_{ii}|\min_{p\in\mathcal{P}_{m-1}} \| (I - p( \Lambda )\Lambda^{-m}(\Lambda - I)) \|_2\|D_{\mu}^{-1}W^*Ae_k\|_2\\
& \leq \max_{1\leq i\leq n} |\Sigma_{ii}| \min_{p\in\mathcal{P}_{m-1}} \max_{x\in[a,b]} \left| (I - p(x )x^{-m}(x - I))\right| \|\Sigma\|_2\|D_{\mu}^{-1}W^*Ae_k\|_2.
\end{aligned} 
\]
The result follows as the polynomial optimization problem is identical to the one in the proof of Theorem~\ref{thm:ErrorOneStepAA}.
\end{proof} 

To get a sense of Theorem~\ref{thm:WeightedAAError}, suppose that $\Sigma_{jj} = 1/j^2$. Then, $\max_{1\leq i\leq n} |\Sigma_{ii}| = 1$ so that the inequalities in Theorem~\ref{thm:ErrorOneStepAA} and Theorem~\ref{thm:WeightedAAError} are almost identical. The only difference is that Theorem~\ref{thm:WeightedAAError} is bounding $\|\Sigma D_{\mu}^{-1}W^*e_{k+1}\|_2$, not $\|D_{\mu}^{-1}W^*e_{k+1}\|_2$. This means that AA in the  weighted $\ell^2$ norm is penalizing the first entry of $D_{\mu}^{-1}W^*e_{k+1}$ more than the last entry. Since $\Sigma D_{\mu}^{-1}W^*e_{k+1}$ contains the term $W^*e_{k+1}$, one can view this as biasing towards certain spectral content of $A$.  

 \begin{figure}
 \centering
 \begin{minipage}{.49\textwidth}
 \includegraphics[width=.9\textwidth]{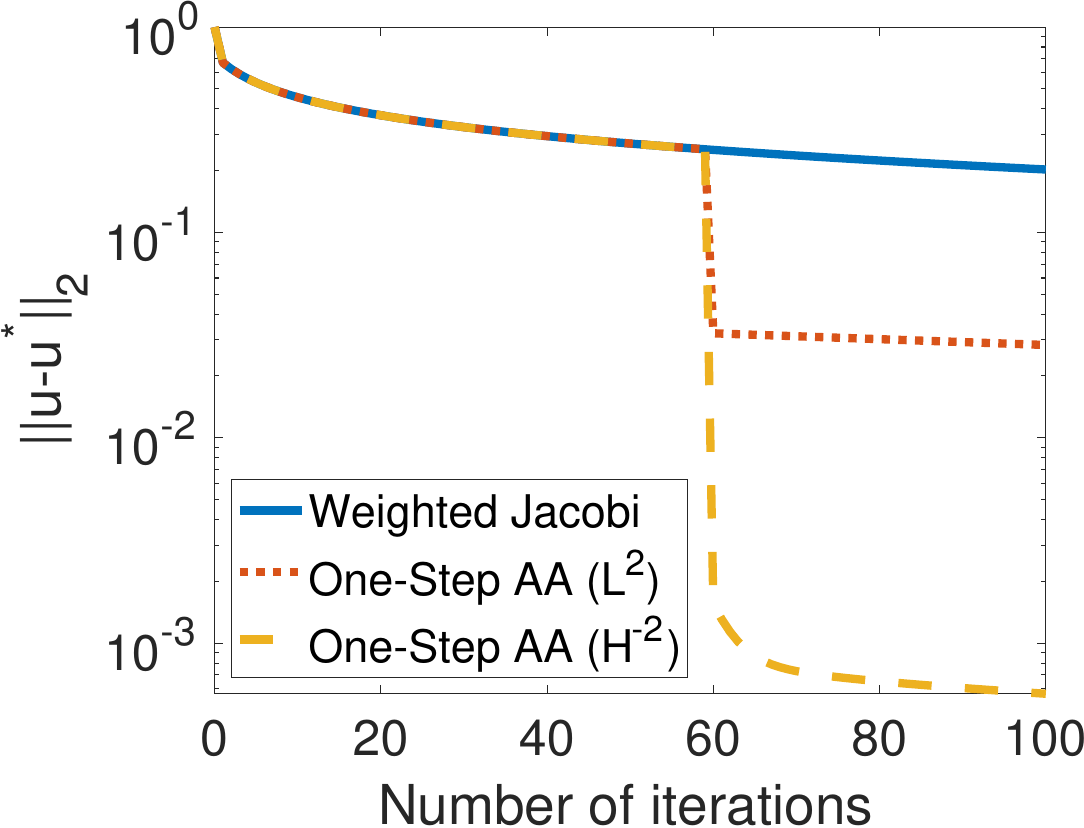}
 \end{minipage}
 \begin{minipage}{.49\textwidth}
 \includegraphics[width=.9\textwidth]{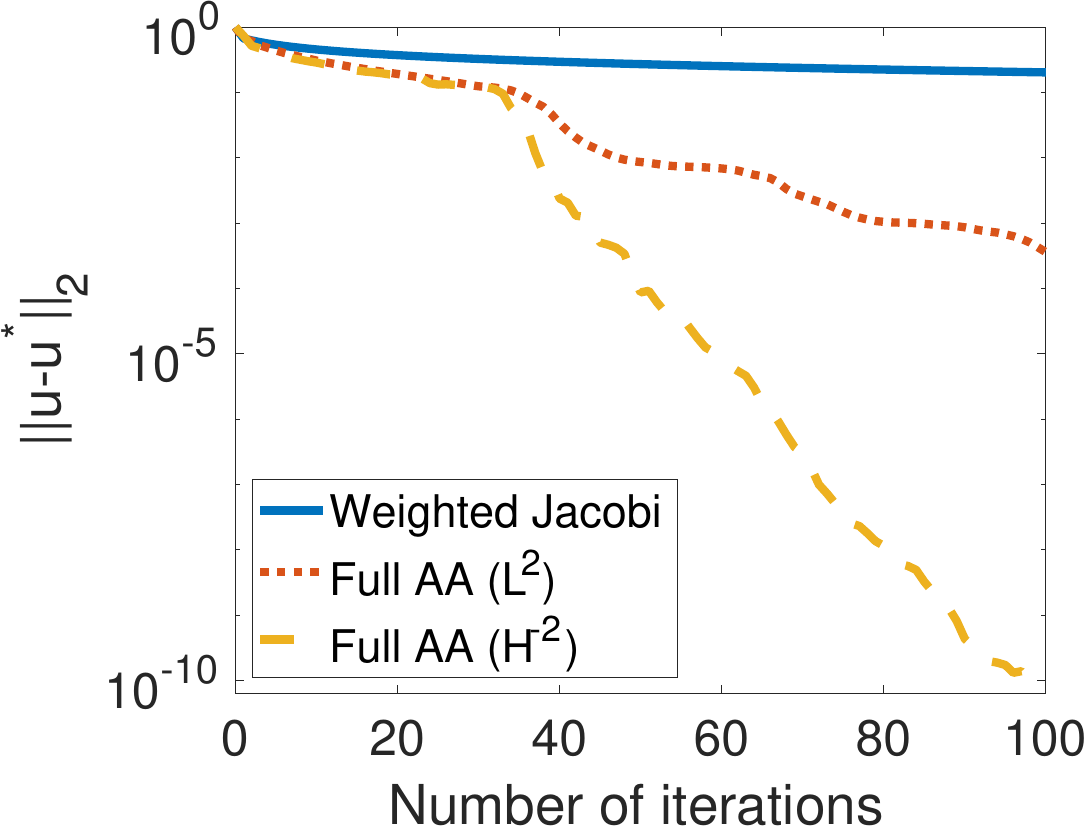}
 \end{minipage}
\caption{Solving Poisson's equation in~\eqref{eq:exp_1D_poisson} with the weighted Jacobi method, AA based on the $L^2$ norm, and AA based on the $\mathcal{H}^{-2}$ norm. Left: The weighted Jacobi method compared to one-step AA. Right: The weighted Jacobi method compared to full AA.}\label{fig:Poisson}
\end{figure}

\section{Numerical experiments} \label{sec:experiments}
In this section, we present numerical experiments to demonstrate the acceleration effects of AA based on the $\mathcal H^{-2}$ norm. We show this with both contractive and noncontractive fixed-point operators that involve second-order differential operators. 

\subsection{A contractive operator for solving Poisson's equation}\label{sec:example1}
Our first example illustrates the theorems in Section~\ref{sec:analysis} and the convergence behavior of AA. For this example, we recommend solving~\eqref{eq:exp_1D_poisson} using direct methods as the linear system is a tridiagonal Toeplitz matrix. We only use this example to illustrate our theorems.

Consider 1D Poisson's equation with zero Dirichlet boundary conditions on $(0,1)$, i.e., 
\begin{equation}~\label{eq:exp_1D_poisson}
-u''(x) = f(x), \qquad u(0) = u(1) = 0. 
\end{equation}
We discretize~\eqref{eq:exp_1D_poisson} by using a second-order finite difference scheme to obtain the $n\times n$ linear system
\begin{equation}
\underbrace{\frac{1}{h^2}\begin{bmatrix}-2 & 1 \cr 1 & \ddots & \ddots \cr & \ddots &\ddots & 1\cr && 1 & -2 \end{bmatrix}}_{=M}\!\! \begin{bmatrix} u_1\\u_2\\\vdots\\u_n \end{bmatrix} = \underbrace{\begin{bmatrix}f(x_1)\\f(x_2)\\\vdots\\f(x_n) \end{bmatrix}}_{=b}, \qquad x_j = jh,\qquad h = \frac{1}{n+1}.  
\label{eq:Tridiagonal} 
\end{equation} 
As our fixed-point iteration, we consider the weighted Jacobi method given by
\[
x_{n+1} = G(x_n) = (I-\frac{2}{3} D_M^{-1}M)x_n + \frac{2}{3}D_M^{-1}b, \qquad D_M = {\rm diag}(M).
\]
Here, $G$ is a contractive operator because $A =I-\frac{2}{3}D_M^{-1}M$ has eigenvalues 
\[
\lambda_j(A) = \frac{1}{3} + \frac{2}{3}\cos\!\left(\frac{j\pi}{n+1}\right),\qquad 1\leq j\leq n,
\]
which satisfy $|\lambda_j(A)|<1$ for $1\leq j\leq n$. For each $j$, the eigenvector corresponding to $\lambda_j(A)$ is also known in closed form as 
\[
(v_j)_i = \sin\!\left(\frac{ij\pi}{n+1}\right), \qquad 1\leq i\leq n, 
\]
and the eigenvector components of the residual corresponding to $\lambda_j(A)$ have size $\mathcal{O}(|\lambda_j(A)|^{k})$ after $k$ iterations~\cite[Chapter 2]{briggs2000multigrid}. If $e_0 = \sum_{j=1}^nc_j v_j$ is the initial error in the weighted Jacobi method, then
\[
e_k = A^ke_0 = \sum_{j=1}^n c_j \lambda_j(A)^k v_j, \qquad f_k = (I-A)e_k = \frac{4}{3} \sum_{j=1}^n  \sin^2\! \left(\frac{j\pi}{2(n+1)}\right) c_j\lambda_j(A)^kv_j.
\]
Thus, one can see a spectral biasing in the weighted Jacobi method: after a few iterations, $\|f_k\|_2$ might be small while $\|e_k\|_2$ is not (due to the $\sin^2((j\pi)/(2(n+1)))$ term). In particular, the eigenvector components associated with large $j$ are more heavily weighted in $f_k$ than in $e_k$. In Fig.~\ref{fig:Poisson}, we show the convergence of the weighted Jacobi method and illustrate its poor convergence. 

A natural idea is to use the choice of distance function in AA to counterbalance the spectral biasing in the weighted Jacobi method. In this case, the $\mathcal{H}^{-2}$ norm is a good choice because the eigenvalues of $P = (I_n-B_n + B_n^2)^{-1/2}$ (see Section~\ref{sec:AAdiscretized}) are given by 
\[
\lambda_j(P) = \left(1 + \frac{4}{h^2}\sin^2\!\left(\frac{\pi(j-1)}{2n}\right) + \frac{16}{h^4}\sin^4\!\left(\frac{\pi(j-1)}{2n}\right)\right)^{-1/2}. 
\]
Therefore, the spectral biasing in the $\mathcal{H}^{-2}$ norm approximately counterbalances the spectral biasing in the residual. This is only heuristic reasoning because the eigenvectors of $P$ and $A$ are not the same. Still, in practice, we observe that AA in the $\mathcal{H}^{-2}$ norm converges rapidly.

In Fig.~\ref{fig:Poisson}, we compare the convergence of the weighted Jacobi method to one-step AA as well as the full AA algorithm. For these tests we use $n = 63$, the memory parameter $m = 10$, and use the initial solution error of  
\[
(e_0)_j = \sum_{i=1}^{20} \sin(2ij\pi),\qquad 1\leq j \leq n.
\] 
As can be seen in the figures, applying AA successively at every iteration is preferred over one-step AA.

\subsection{Noncontractive operator for solving Poisson's equation}
One can also attempt to solve~\eqref{eq:exp_1D_poisson} using Richardson iteration. That is, 
\[
x_{n+1} = G_R(x_n) = (I-M)x_n + b.
\]
Now, the fixed-point operator, $G_R$, is noncontractive as there are eigenvalues of $I-M$ larger than one in absolute value. Generally speaking, the Richardson iteration computes a divergent sequence. Nonetheless, to illustrate the surprising acceleration effects of AA, we repeat the experiment from Section~\ref{sec:example1} with the weighted Jacobi method replaced by Richardson iteration. 
 \begin{figure}
 \centering
 \includegraphics[width=0.45\textwidth]{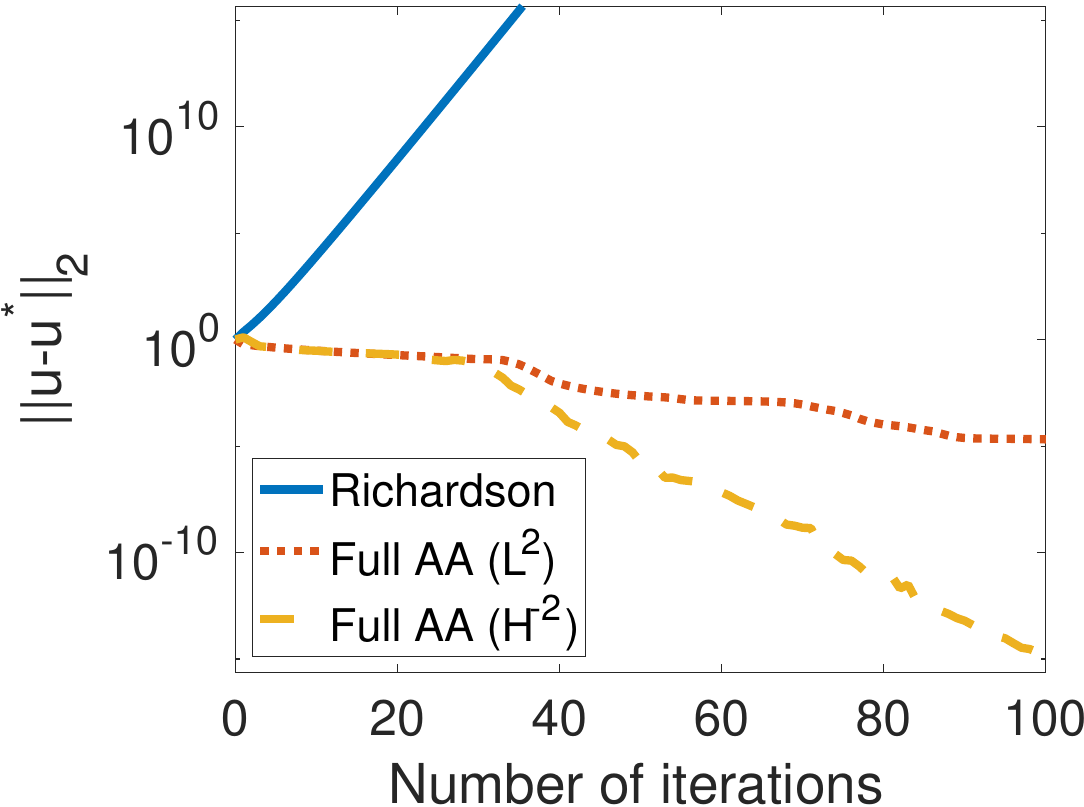}
 \caption{Solving Poisson's equation in~\eqref{eq:exp_1D_poisson} with Richardson iteration, AA based on the $L^2$ norm, and AA based on the $\mathcal{H}^{-2}$ norm. This experiment highlights that AA can be used for noncontractive fixed-point operators.}
 \label{fig:Richardson} 
 \end{figure}

\subsection{Nonlinear Helmholtz equation} 
AA is highly useful for nonlinear fixed-point operators. In this example, we present a nonlinear fixed-point operator designed to find $u:[0,1]\rightarrow \mathbb{C}$, which solves the following 1D nonlinear Helmholtz equation~\cite{pollock2019anderson}
\begin{equation}
    \begin{aligned} 
\dfrac{d^2 u}{dx^2} + k^2_0(1+\epsilon(x) |u|^2) u &=0,\quad 0<x<1,\\
\dfrac{du}{dx} +ik_0 u &=2ik_0,\quad x=0, \\
\dfrac{du}{dx} -ik_0 u &=0,\quad x=1.\label{eq:NHL}
\end{aligned} 
\end{equation} 
The nonlinear Helmholtz equation governs the propagation of linearly-polarized, time-harmonic electromagnetic waves in Kerr-type dielectrics~\cite{baruch2007high}. We set $\epsilon(x)$ to be a piecewise constant function on $[0,1]$, which approximates a realistic
grated Kerr medium~\cite[p.~3]{baruch2007high}: 
\[
  \epsilon(x)=\begin{cases}
               0, &  0\leq x \leq 0.1, \\
               1, &  0.1< x \leq 0.2, \\
               2, &  0.2< x \leq 0.3, \\
               3, & 0.3< x \leq 0.7, \\
               4, & x>0.7.\\
            \end{cases}
\]
In the numerical tests, the system~\eqref{eq:NHL} is discretized by the same second-order finite difference method as described in~\cite[Sec.~6.2]{pollock2019anderson}. The resulting iterative scheme can be seen as a fixed-point operator: $u_{k+1} = G_{\text{NHL}}(u_k)$. Following~\cite{pollock2019anderson}, we set the initial guess to be $u_0 = e^{ik_0x}$, where $k_0$ is the linear wavenumber and $x$ is the discretized interval $[0,1]$ with grid spacing $h = 0.002$.

 \begin{figure}
 \centering
 \includegraphics[width=0.49\textwidth]{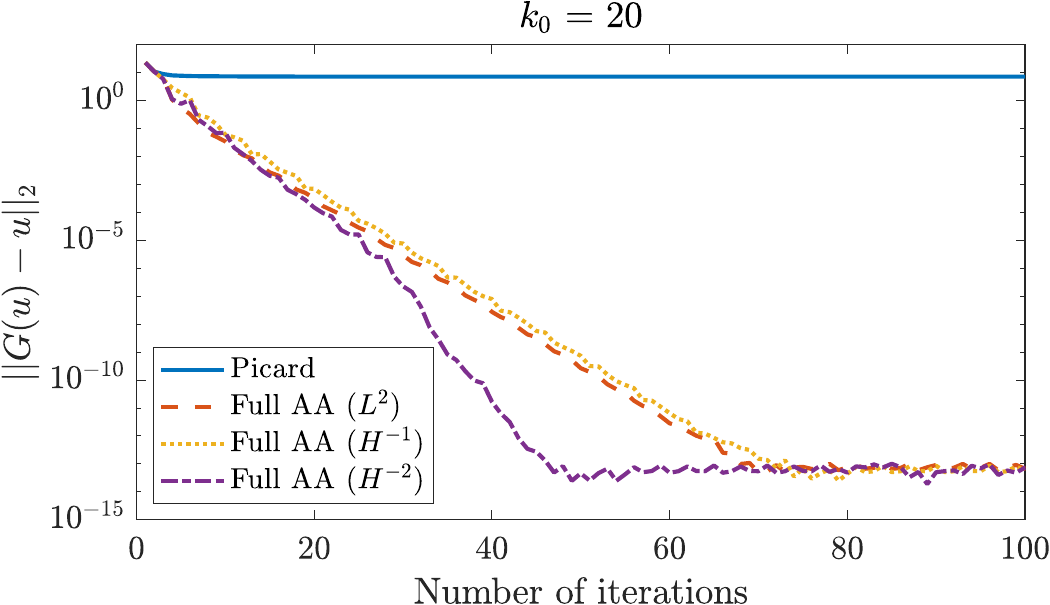}
 \includegraphics[width=0.49\textwidth]{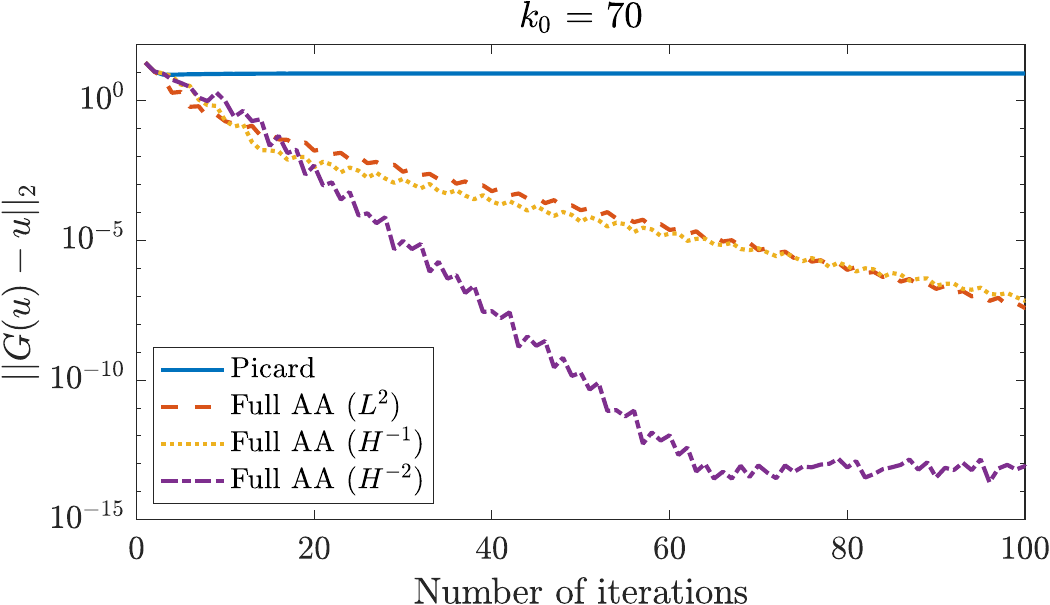}
 \caption{Solving the nonlinear Helmholtz equation by Picard iteration, AA with $m=1$ using the $L^2$ norm, the $\mathcal H^{-1}$ norm, and the $\mathcal H^{-2}$ norm with $k_0 = 20$ (left) and $k_0 = 70$ (right). The residuals are presented for the first $100$ iterations.}\label{fig:NLH} 
 \end{figure} 
 
Fig.~\ref{fig:NLH} shows the numerical results of solving the nonlinear Helmholtz equation with $k_0 = 20$ (left) and $k_0= 70$ (right), respectively. In both cases, the Picard iteration fails to converge, and the residual remains constant for $100$ iterations, while AA in the $L^2$, $\mathcal  H^{-1}$, and $\mathcal  H^{-2}$ norm all decrease the residual rapidly. As the wavenumber $k_0$ increases, the nonlinear Helmholtz problem becomes more challenging, and $\mathcal{H}^{-2}$ norm becomes more beneficial. 
 
 Although~\eqref{eq:NHL} is nonlinear, there is spectral biasing from the second-order spatial derivative.  Thus, one expects that the $\mathcal H^{-2}$ norm could be counterbalancing the spectral bias of the fixed-point operator in~\eqref{eq:NHL}. Nonetheless, the spectral properties might change drastically with different $\epsilon(x)$, so one must be careful. Also, we note that the convergence behavior highly depends on the initial guess $u_0$. We also observe interesting convergence behavior as $m$ increases. In particular, there seems to be essentially no benefit in taking large $m$, which we believe could be related to the nonlinearity in~\eqref{eq:NHL}.

\begin{figure}
\begin{center}
\includegraphics[width=0.30\textwidth]{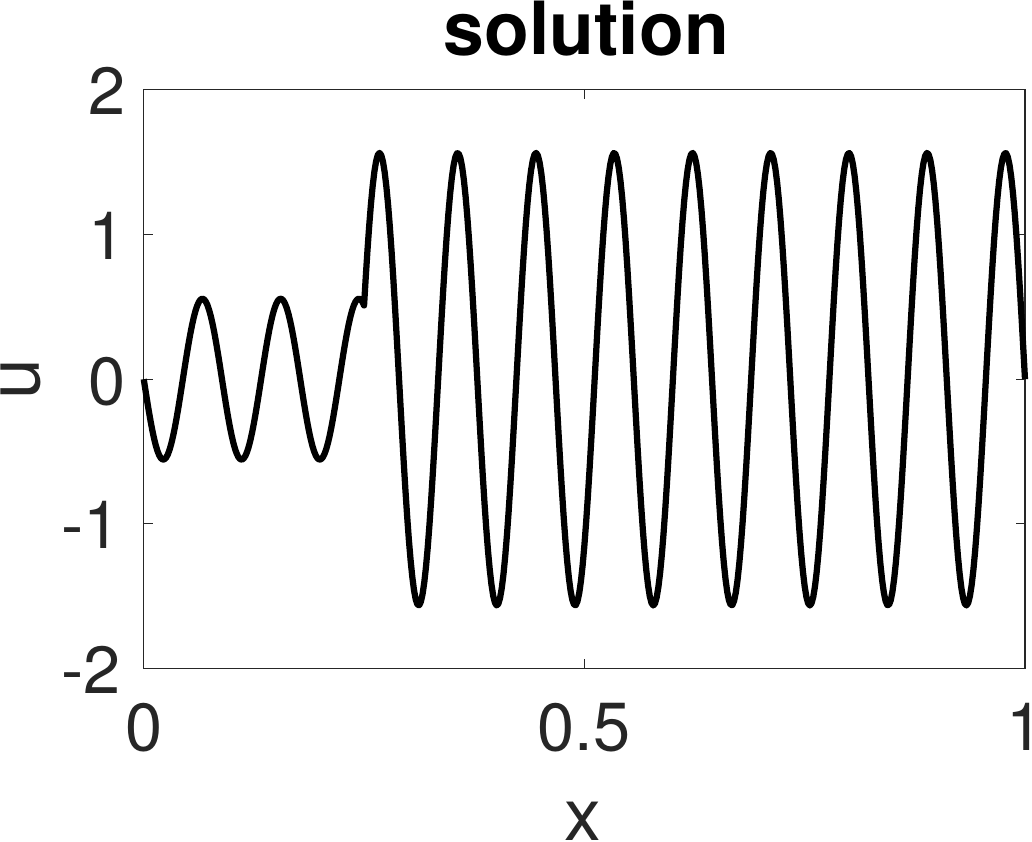}
\includegraphics[width=0.30\textwidth]{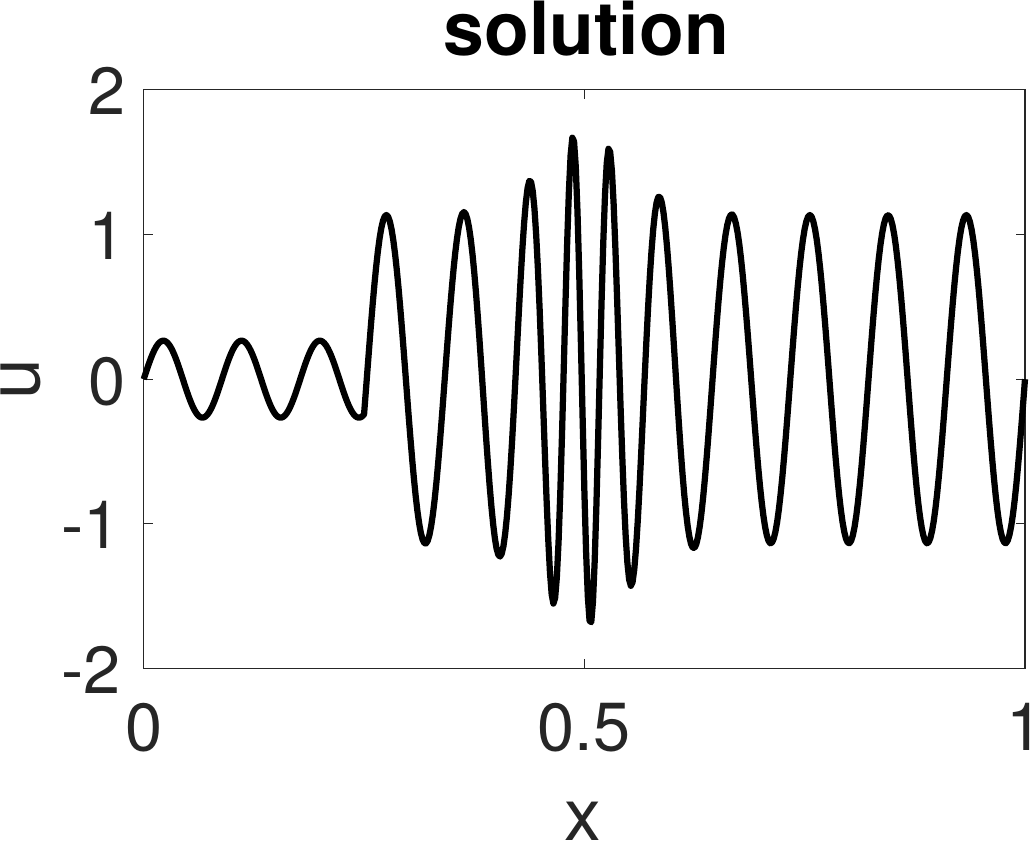}
\includegraphics[width=0.30\textwidth]{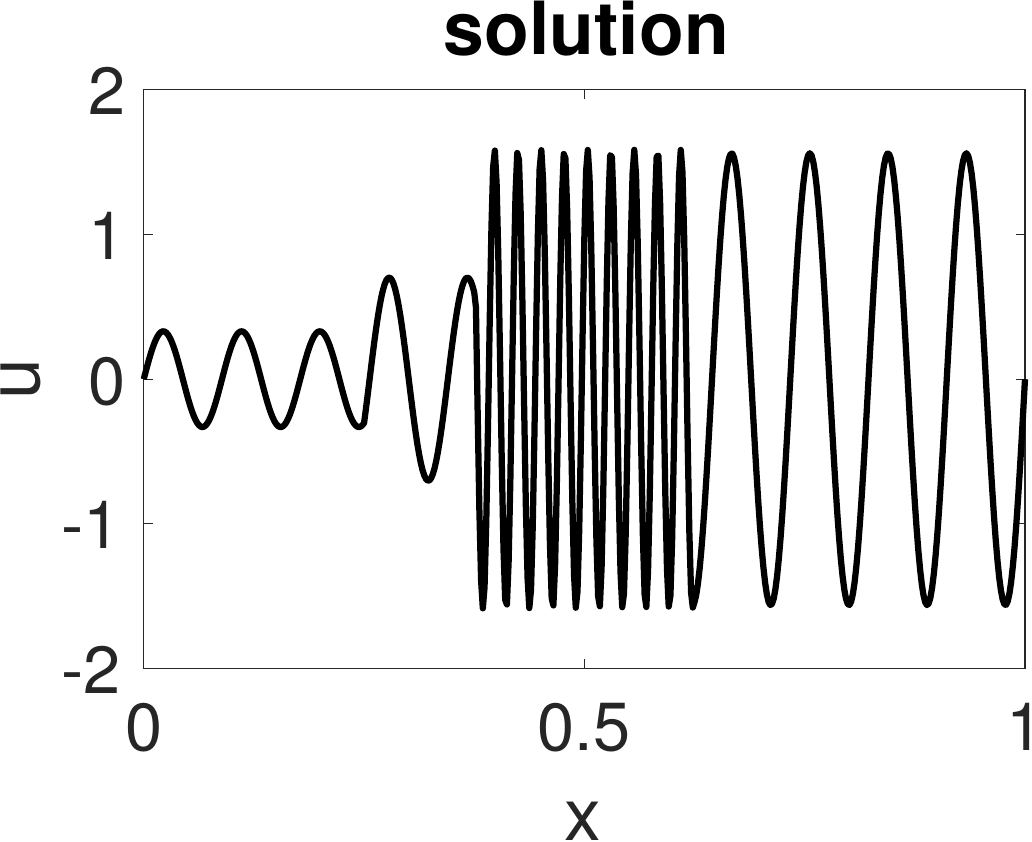}
\caption{Solutions for the three different wave speeds $c_{\rm a}$ (left), $c_{\rm b}$ (middle) and $c_{\rm c}$ (right); see~\eqref{eq:cvalues}. \label{fig:whi2}}
\end{center}
\end{figure}

\subsection{Solving the Helmholtz equation using the WaveHoltz iteration} \label{sec:whi}
In this experiment, we explore AA for the WaveHoltz iteration~\cite{appelo_garcia_runborg_WHI}. In the WaveHoltz iteration, we consider the Helmholtz equation in a bounded open Lipschitz domain $\Omega$, i.e., 
\begin{equation}
  \nabla\cdot(c^2(x)\nabla u) + \omega^2u = f(x),\qquad
  x\in \Omega, \label{eq:helm}
\end{equation}
together with the energy conserving homogeneous Dirichlet boundary conditions. That is, 
\begin{equation}
i\omega u = 0,
\qquad x\in \partial\Omega.
 \label{eq:helmbc}
\end{equation}
As a result, the solution to~\eqref{eq:helm} with the boundary condition~\eqref{eq:helmbc} is a real-valued function. To find a solution to~\eqref{eq:helm}-\eqref{eq:helmbc}, we use the fixed-point iteration given by 
\[
 u_{k+1} = \Pi u_{k}, \qquad u_{0}\equiv 0,
\]
where (see~\cite{appelo_garcia_runborg_WHI})
\begin{equation} 
  \Pi u = \frac{2}{T}\int_0^{T}\left(\cos(\omega t)-\frac14\right)w(t,x) dt,\qquad
  T=\frac{2\pi}{\omega}. 
  \label{eq:projection} 
  \end{equation}
Here, $w(t,x)$ depends on $u(x)$ via the following wave equation:
\begin{equation} 
\begin{aligned} 
w_{tt} - \nabla\cdot(c(x)^2\nabla w)
  & = f(x) \cos(\omega t), \quad x\in\Omega, \ \ 0 \le t \le T,  \\
 w(0,x) &= u(x), \quad w_t(0,x) = 0,\quad \quad x\in\Omega,\\
w(t,x) &= 0, \quad x\in\partial\Omega, \quad 0 \le t \le T.
    \end{aligned} 
        \label{eq:wave}
\end{equation}

First, we take $c(x)$ to be a variable wave speed and $\Omega  = [0,1]$. We discretize the problem with an equispaced grid $x_j = j h$ for $0\leq j\leq n$ and $h = 1/(n+1)$. We approximate the wave equation in~\eqref{eq:wave} using a second-order finite difference scheme in space and time. If one takes $w^k_i \approx w(t_k,x_i)$, then we have the following discretization: 
\begin{eqnarray*}
&& w_i^{-1}  = v^n_i + \frac{\Delta t^2}{2} \left( \frac{(c_{i+1}+c_{i})v^{n}_{i+1}- (c_{i+1}+2c_{i}+c_{i-1})v^{n}_{i}+(c_{i}+c_{i-1})v^{n}_{i-1}}{2h^2} - f(x_i) \right), \ \ \forall i, \\
&& w_i^0 = v^n_i, \ \ \forall i, 
\end{eqnarray*}
\begin{eqnarray*}
\frac{w^{k+1}_i - 2 w^{k}_i+ w^{k-1}_i}{\Delta t^2} &=& 
\frac{(c_{i+1}+c_{i})w^{k}_{i+1}- (c_{i+1}+2c_{i}+c_{i-1})w^{k}_{i}+(c_{i}+c_{i-1})w^{k}_{i-1}}{2h^2} \\
&&- \cos(\omega t_k)f(x_i), \quad  i = 1,\ldots,n-1, \quad k \ge 0. \\
w^{k}_{i} &=& 0, \quad i = 0, \quad i = n, \quad \forall k.
\end{eqnarray*}
Here, the integral in the projection~\eqref{eq:projection} is discretized by the trapezoidal rule. Since this is a linear problem, the discretized solution to the nonlinear Helmholtz equation in~\eqref{eq:helm} can be shown to solve a linear system, i.e., $Au = b$.


Denoting the eigenpairs of the operator $u\mapsto -\nabla\cdot c(x)^2\nabla u$ by $(\lambda_i,\phi_i(x))$, the eigenvalues of the fixed-point iteration satisfy the implicit relationship
\[
\beta(\lambda_i) =  \frac{2}{T}\int_0^{T}\left(\cos(\omega t)-\frac14\right) \cos(\lambda_i t), 
\]
where $\beta(\lambda)$ is a function such that $\beta(\lambda)\in [-1/2,1)$. For the discretized problem one can show that $A$ has eigenvalues in the interval $(0,3/2]$ and shares its eigenvectors with the tridiagonal matrix in~\eqref{eq:Tridiagonal}.   

In our numerical experiments, we use $n = 513$ and set $\omega =  25\sqrt{2}$. The forcing term is zero everywhere, except at the $128$th gridpoint. At the $128$th gridpoint, we set it to be a constant scaled so that the numerical solution is $\approx 2$ in magnitude. The solutions are displayed in Fig.~\ref{fig:whi2}.

\begin{figure}
\begin{center}
\includegraphics[width=0.32\textwidth]{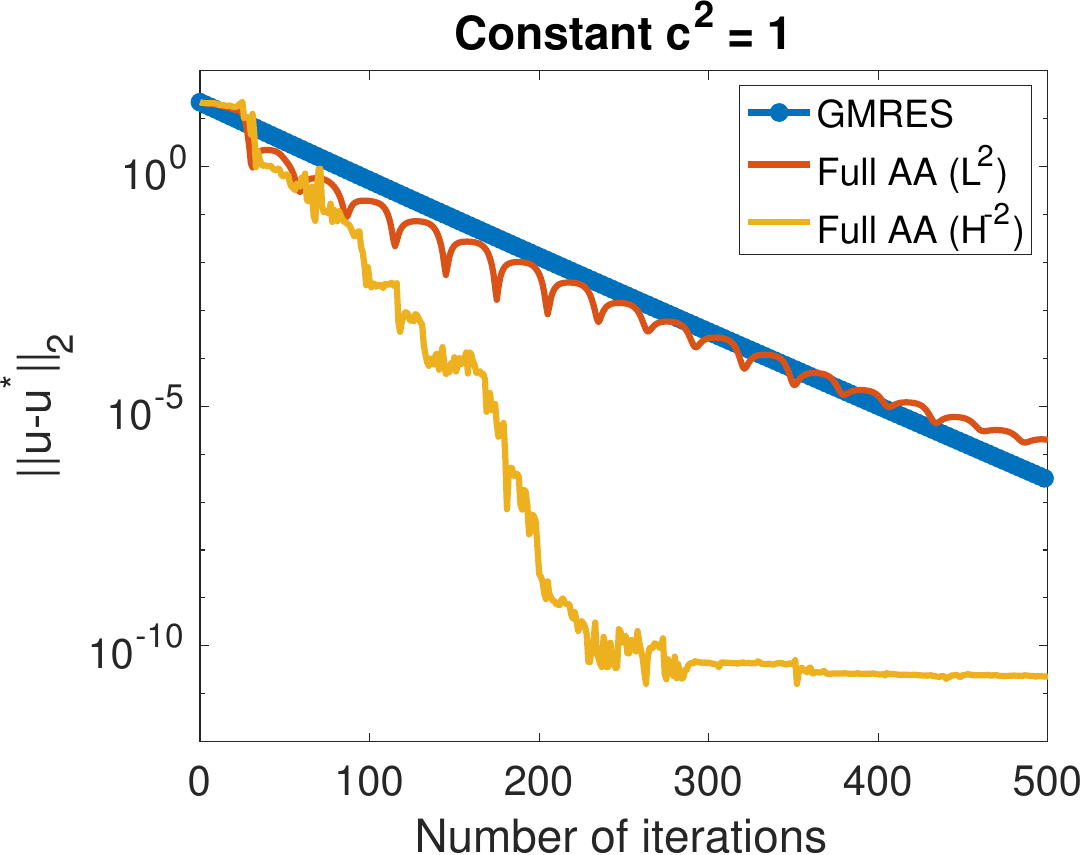}
\includegraphics[width=0.32\textwidth]{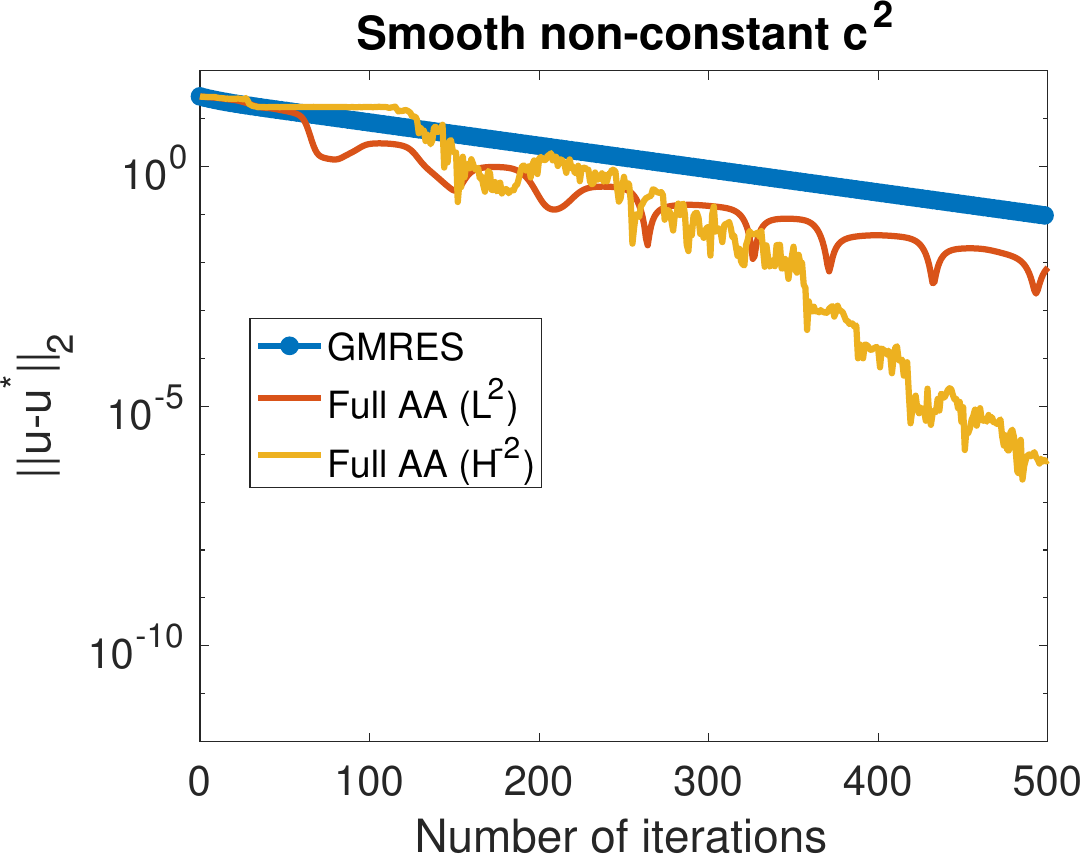}
\includegraphics[width=0.32\textwidth]{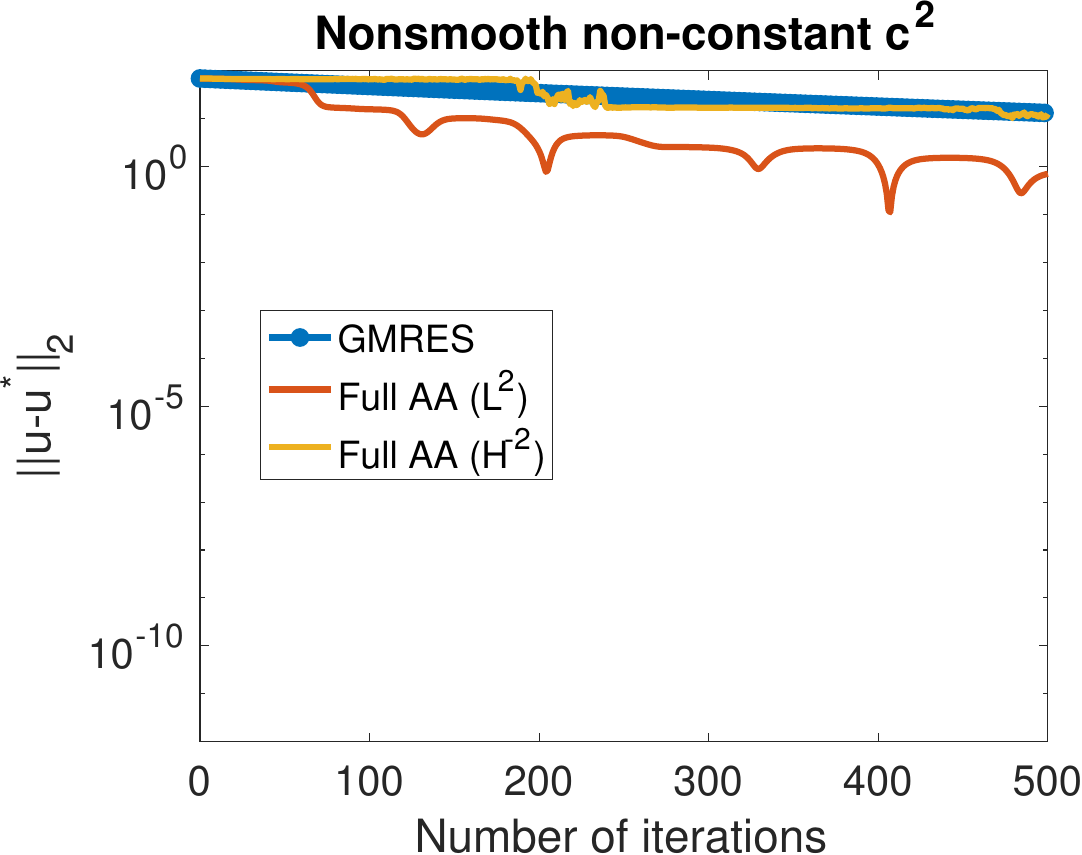}
\smallskip
\smallskip
\includegraphics[width=0.32\textwidth]{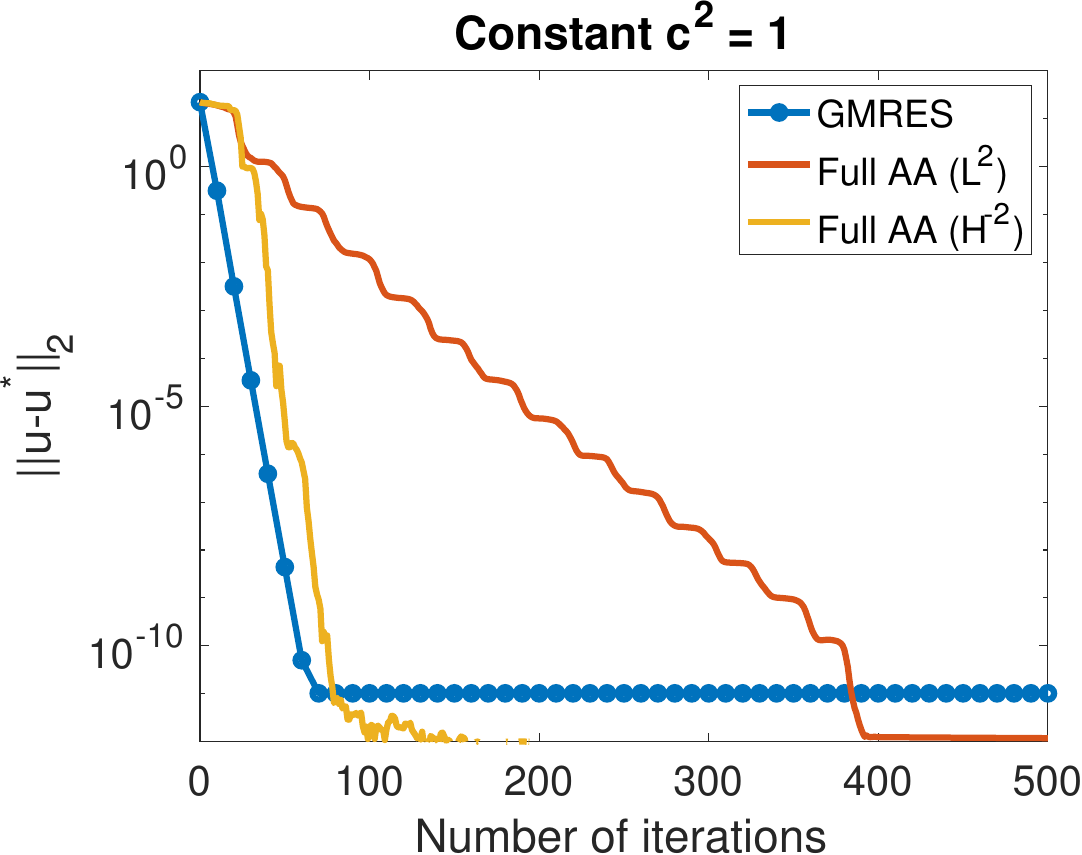}
\includegraphics[width=0.32\textwidth]{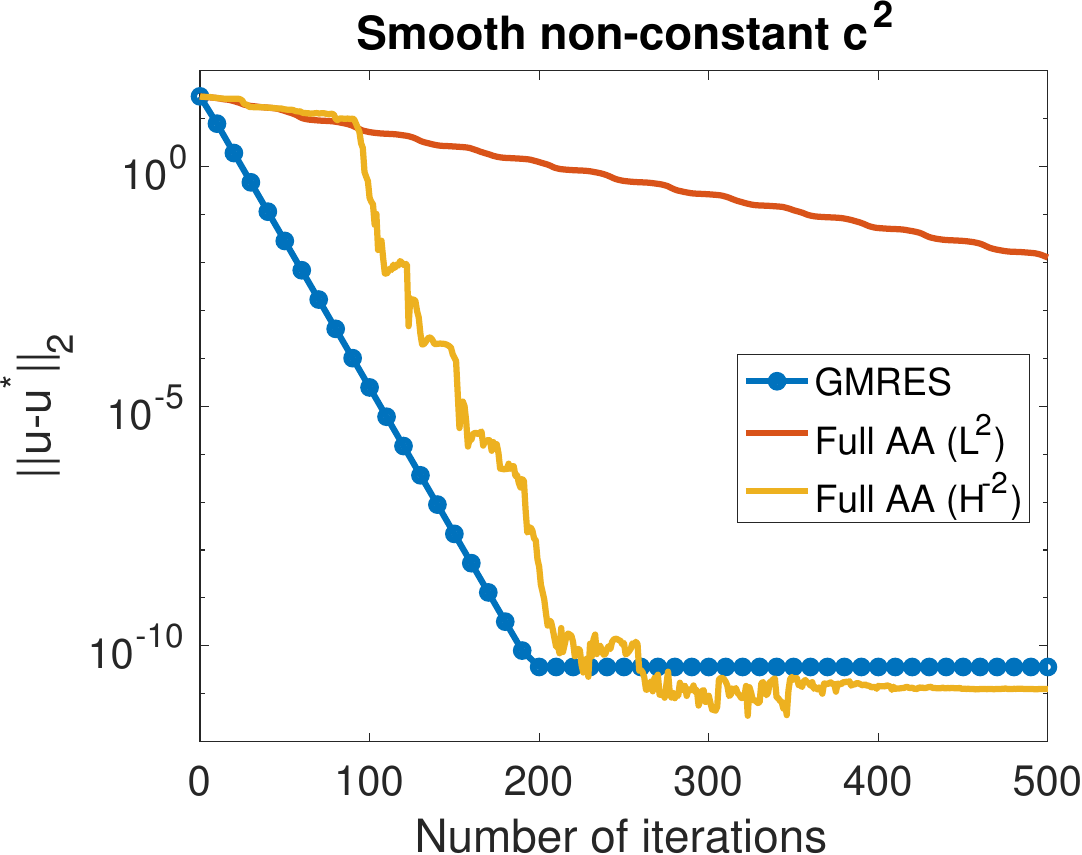}
\includegraphics[width=0.32\textwidth]{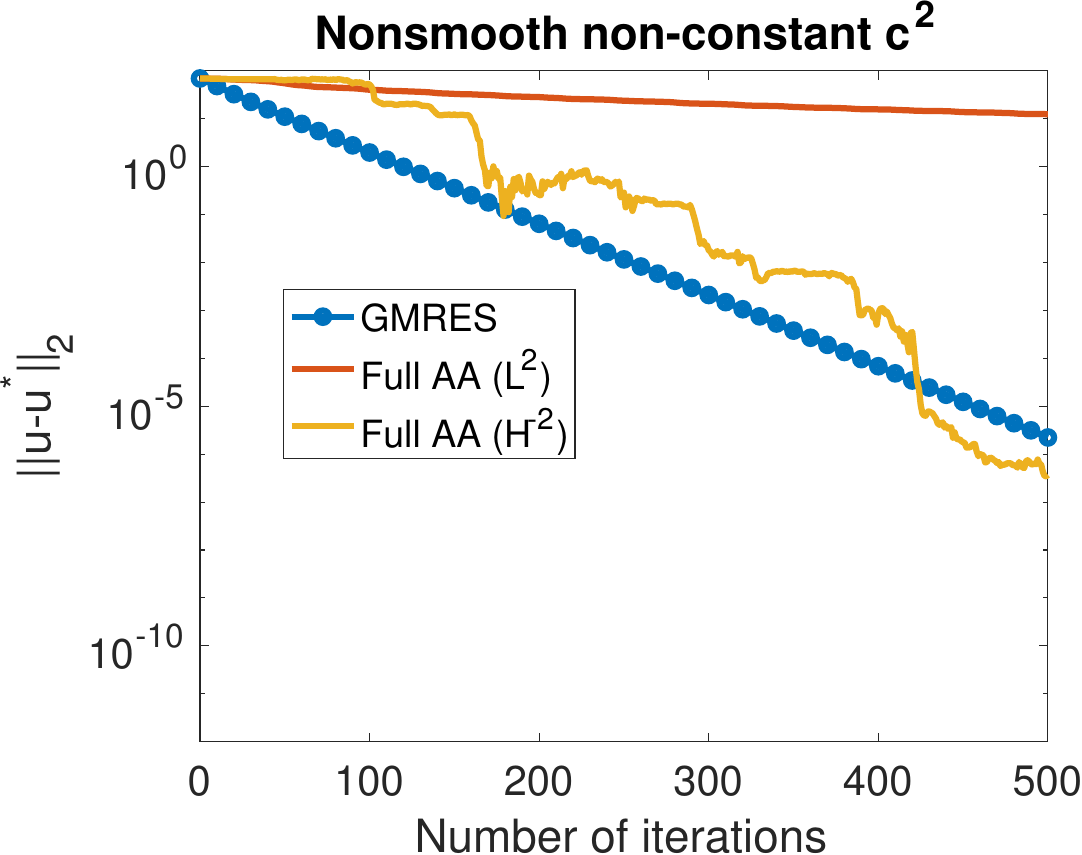}
\smallskip
\smallskip
\caption{Comparison of AA in the $L^2$ and $\mathcal{H}^{-2}$ norm with GMRES. Top row ($m = 3$) and bottom row ($m = 10$). Left column, wave speed $c_{\rm a}$, middle column, wave speed $c_{\rm b}$, and bottom column, wave speed $c_{\rm c}(x)$.
\label{fig:whi1}}
\end{center}
\end{figure}

To compare AA in the $L^2$ norm with AA in the $\mathcal H^{-2}$ norm, we consider three different wave speeds 
\begin{equation} 
c_{\rm a} = 1, \qquad  c_{\rm b} = 1 - 0.55 e^{-144(x-0.5)^2}, \qquad 
c_{\rm c} = \left\{ \begin{array}{cc}
1 & |x-0.5| >  0.125, \\
 0.3 & |x-0.5| < 0.125.
\end{array} \right.
\label{eq:cvalues} 
\end{equation} 
We supply AA with the fixed-point operator associated to $Au = b$ obtained by performing Richardson iteration.  We also compare with restarted GMRES, which restarts every $m$ iterations (see Fig.~\ref{fig:whi1}). 
We find that AA in the $\mathcal H^{-2}$ norm outperforms GMRES as well as AA in the $L^{2}$ norm for all cases except when the wave speed is $c_{\rm c}$ and $m=3$. In this exceptional case, none of the methods manage to decrease the error substantially within 500 iterations.


\begin{figure}
\begin{center}
\includegraphics[width=0.45\textwidth]{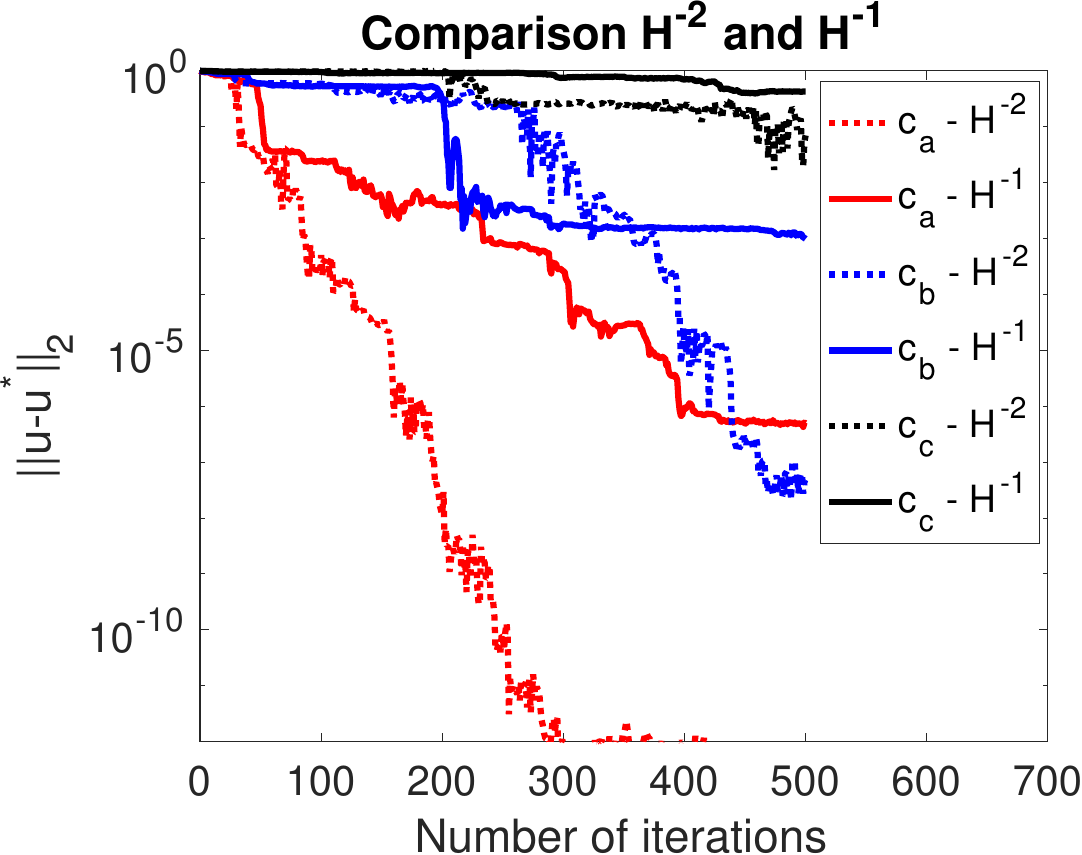}
\includegraphics[width=0.45\textwidth]{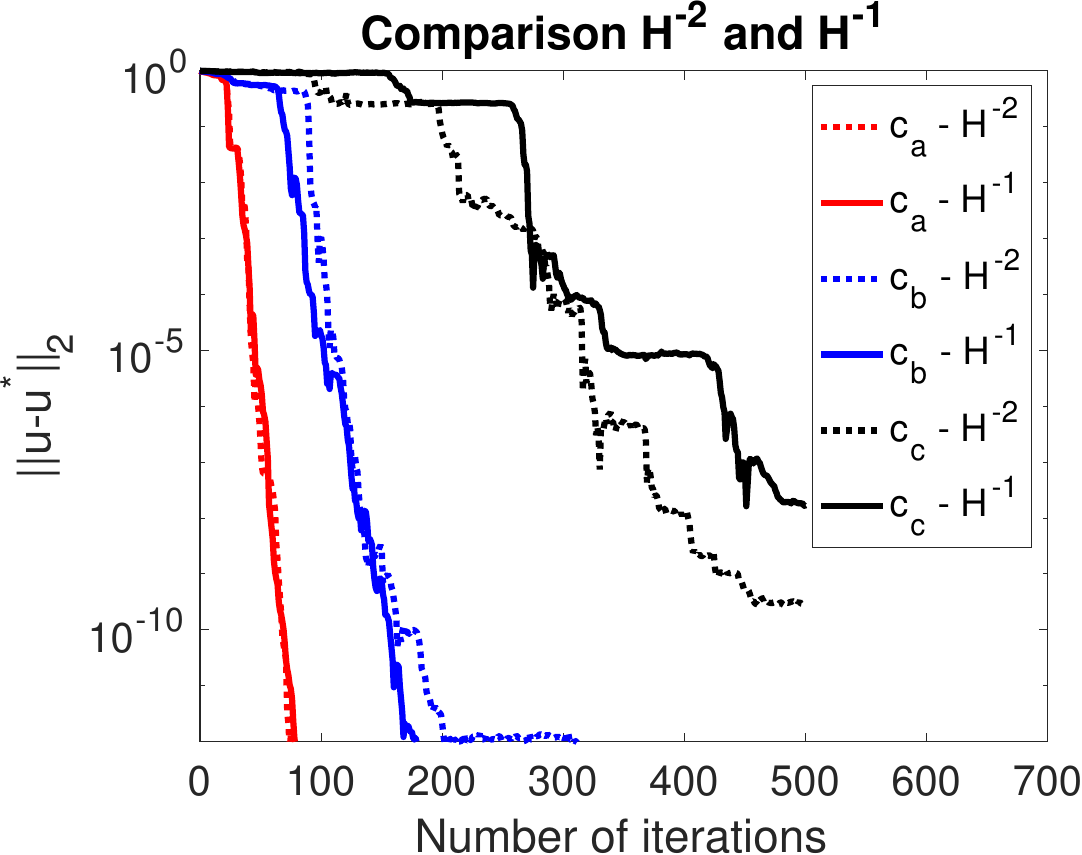}\\
\smallskip
\smallskip
\smallskip
\caption{Comparison of AA in the $\mathcal H^{-1}$ norm and the $\mathcal H^{-2}$ norm for $m = 3$ (left) and $m = 10$ (right).\label{fig:whi1D_H1_H2}}
\end{center}
\end{figure}

The theory for the WaveHoltz iteration predicts that components of the solution that correspond to eigenpairs with eigenvalues closest to the Helmholtz frequency, $\omega$, have the slowest converge rate. When the numerical solution has a small discretization error, these modes are typically well-resolved, and the spectral biasing due to the $\mathcal H^{-2}$ norm appears to improve the convergence of AA. To compare the $\mathcal{H}^{-1}$ and $\mathcal{H}^{-2}$ norm more closely in AA, we compare the acceleration methods for $m=3$ and $m=10$ with the three different wave speeds (see Fig.~\ref{fig:whi1D_H1_H2}). For small $m$, AA in the $\mathcal H^{-2}$ norm converges faster than in the $\mathcal H^{-1}$ norm.

\subsubsection{WaveHoltz iteration in two dimensions}
Next we consider an example in two dimensions with a wave speed given by 
\[
c^2(x,y) = \begin{cases} 
0.3, & 0.4 \le x \le 0.6 \text{ or } 0.4\le  y \le 0.6, \\
1, & \text{otherwise},
 \end{cases} 
\]
on $\Omega  = [0,1]^2$. We, again, use second-order finite difference discretization with an $65\times 65$ equispaced grid. We set the angular frequency to be $\omega = 11$. The forcing term is zero everywhere except at the gridpoint closest to $(0.25,0.75)$ where it has an amplitude adjusted so that the solution is around 1 in magnitude. 

\begin{figure} 
\begin{center}
\subfloat[$m=30$]{\includegraphics[width=0.33\textwidth]{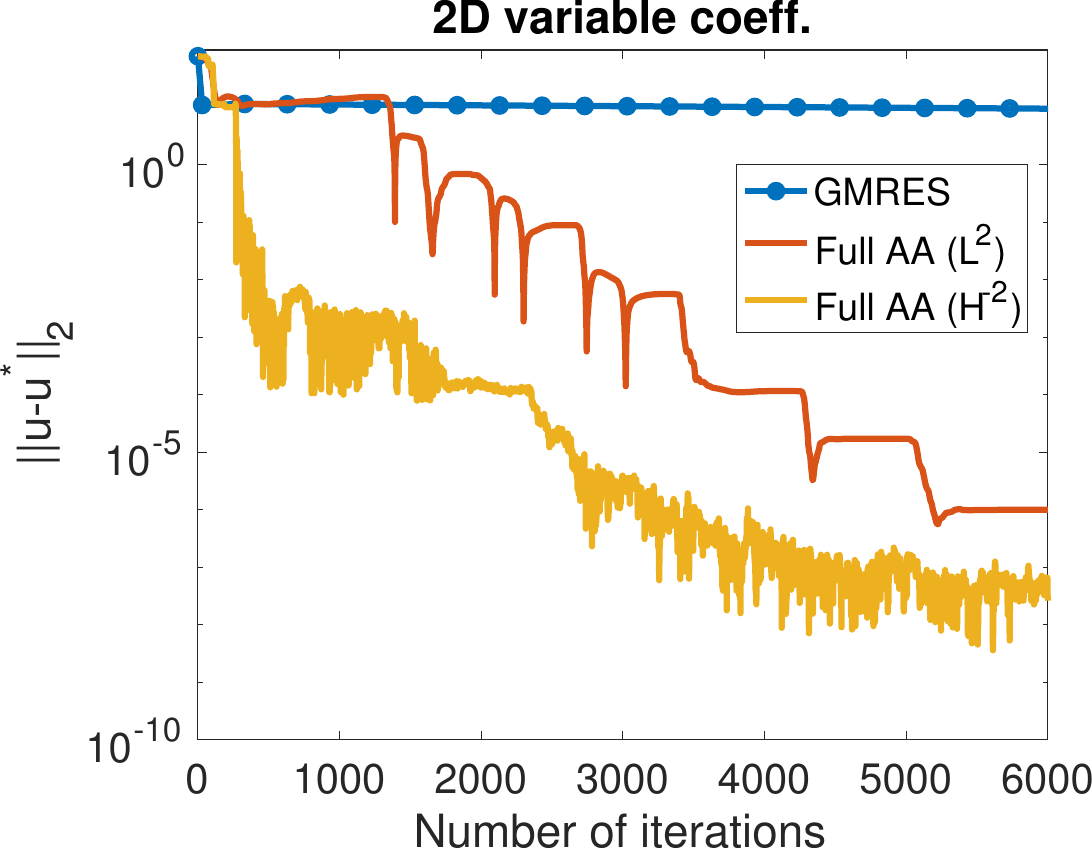}}
\subfloat[$m=50$]{\includegraphics[width=0.33\textwidth]{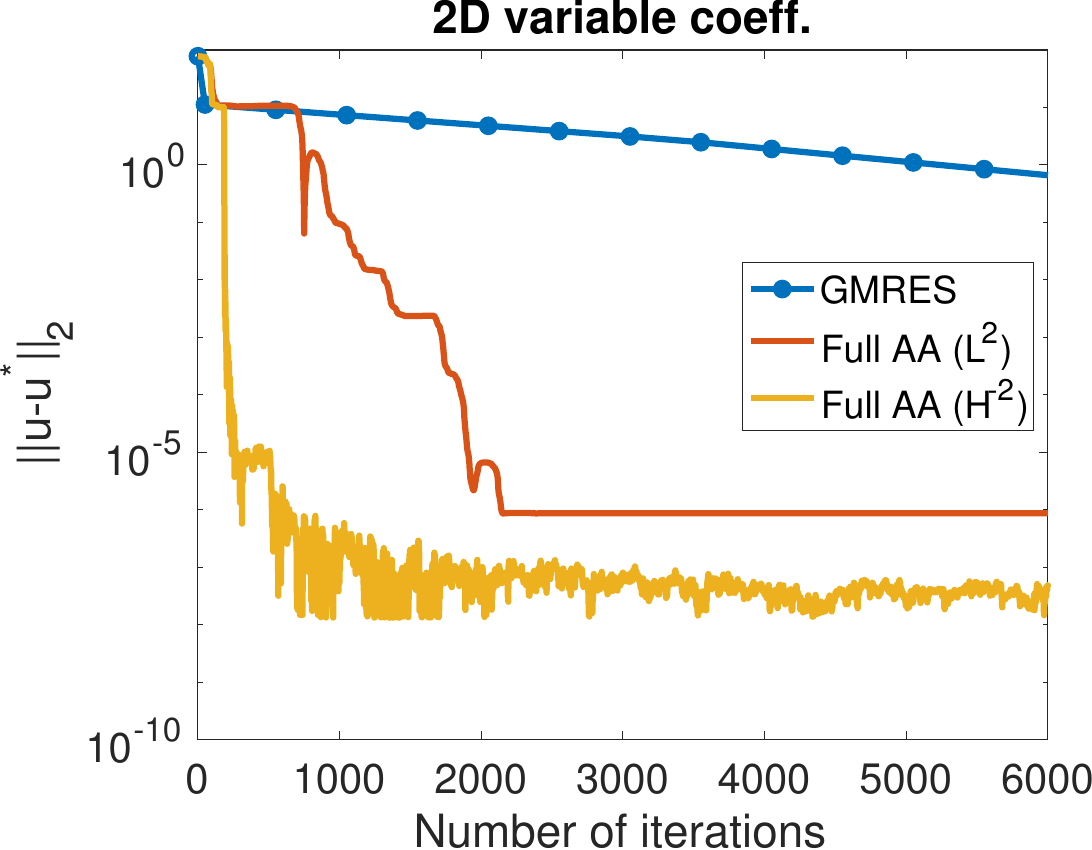}}
\subfloat[$m=100$]{\includegraphics[width=0.33\textwidth]{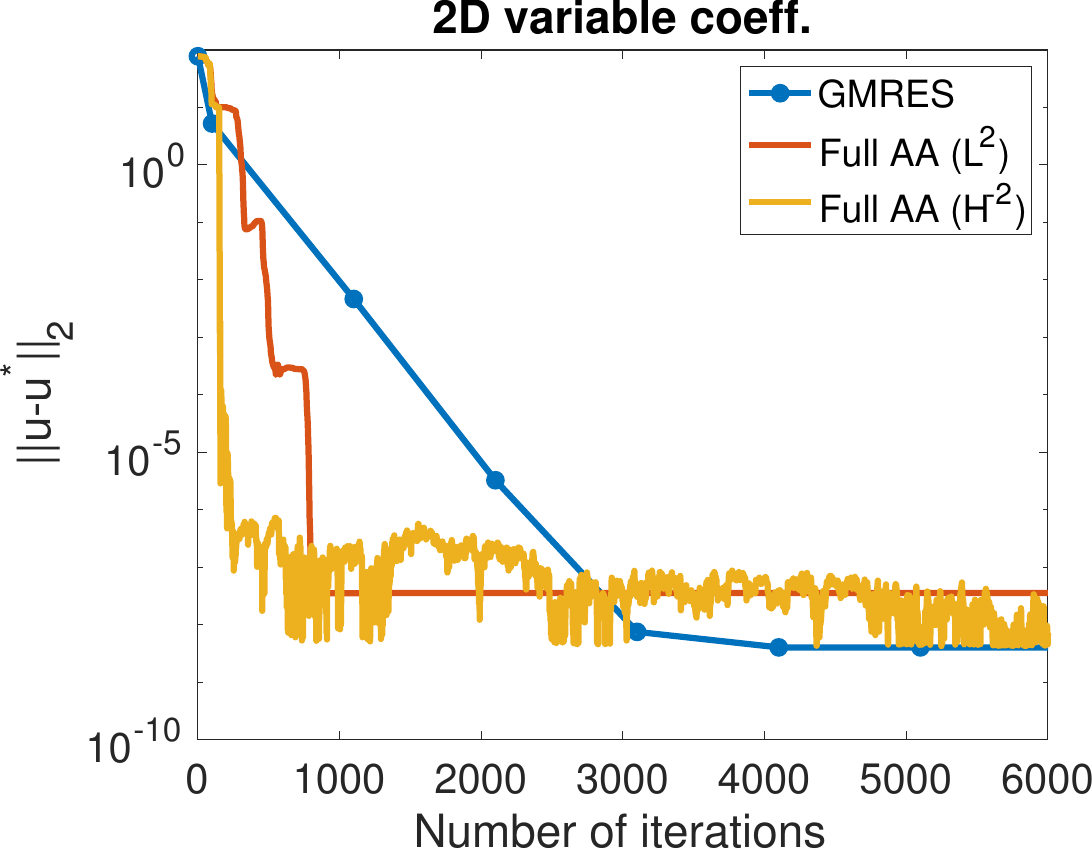}}
\caption{Convergence histories for the two-dimensional Helmholtz example with (a) $m=30$, (b) $m=50$ and (c) $m=100$.\label{fig:whi2D}}
\end{center}
\end{figure}

We employ Richardson iteration to generate a fixed-point operator and then use AA in the $L^2$ and $\mathcal{H}^{-2}$ norm. Note that the degrees of freedom of the solution in this example are ordered as a long vector with lexicographical ordering and that we employ the one dimensional $\mathcal{H}^{-2}$ norm to the two-dimensional data. We compare against restarted GMRES with restarts every $m$th iteration for $m = 30,$ $m= 50$ and $m = 100$ (see Fig.~\ref{fig:whi2D}). We observe that AA in the $\mathcal H^{-2}$ norm outperforms restarted GMRES as well as AA in the $L^{2}$ norm. In this 2D setting, AA in the $L^{2}$ norm is also performing better than GMRES. Not surprisingly, the advantage of the ``sliding memory'' of AA is reduced as the restart depth increases. The solution to the problem is displayed in Fig.~\ref{fig:whi2DSol}.

\begin{figure} 
\begin{center}
\includegraphics[width=0.5\textwidth]{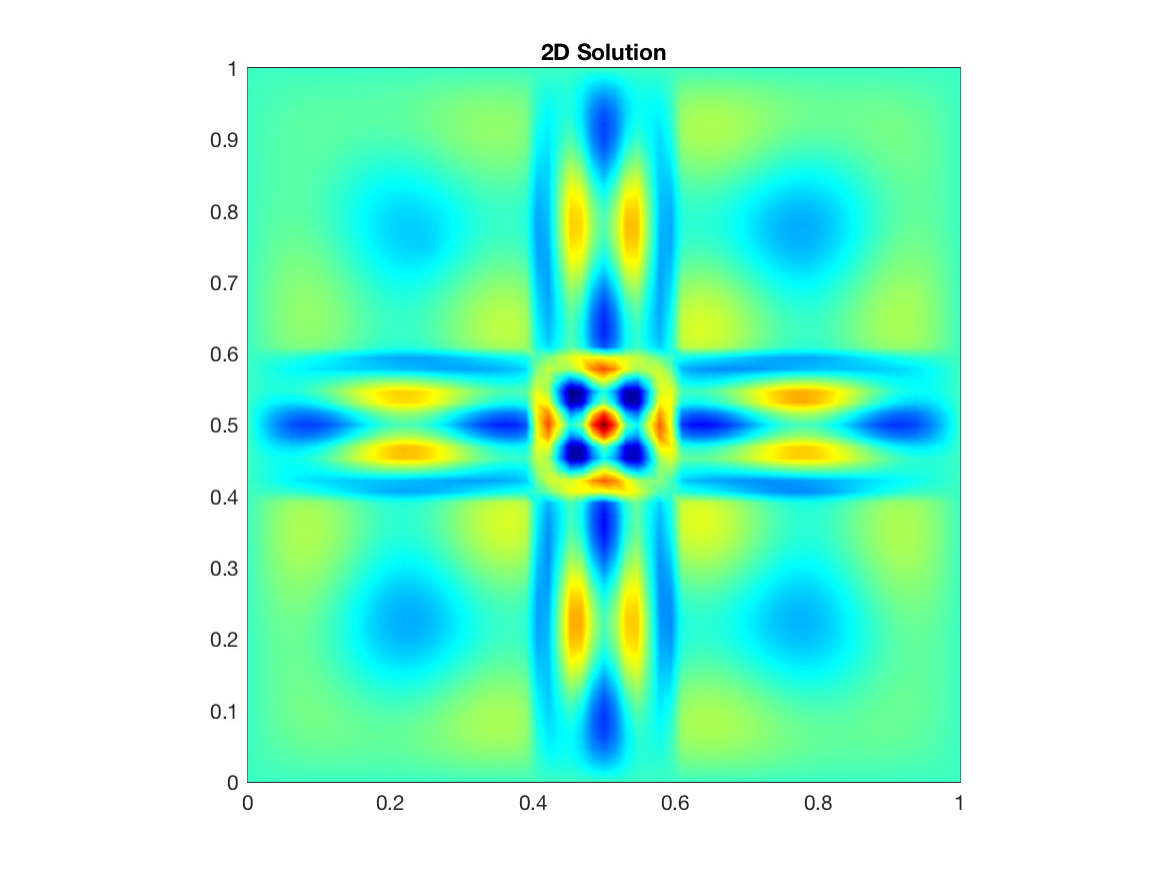}
\caption{The solution for the Helmholtz equation in two dimensions.\label{fig:whi2DSol}}
\end{center}
\end{figure}


This experiment is encouraging, as the benefits of the $\mathcal{H}^{-2}$ norm persist. Even though it might be possible to use a two-dimensional definition of the $\mathcal H^{-2}$ for this simple geometry, its computation can be more costly and, in the case of complex geometry, quite cumbersome to compute.  

\section{Conclusion}\label{sec:Conclusions}
In this paper, we propose the idea of using Anderson acceleration based on the $\mathcal H^{-s}$ Sobolev norm. We observe that this can counterbalance the implicit spectral biasing in certain fixed-point operators. We rigorously analyze the convergence behavior of one-step AA, providing an explicit error bound using Chebyshev polynomials that decreases exponentially in the memory parameter $m$. Numerical experiments for both contractive, noncontractive, and non-linear operators demonstrate the acceleration effects of AA based on different norms. In practice, the choice of distance function in AA should be selected depending on the spectral biasing of the fixed-point operator. 

\section*{Acknowledgments}
This material is based upon work supported by the National Science Foundation under Grant No. DMS-1818757, DMS-1913076 and DMS-1913129.


\end{document}